\documentclass[10pt,a4paper]{article}
\usepackage[utf8]{inputenc}
\usepackage[english]{babel}

\pdfoutput=1

\usepackage{amsmath,amsfonts,amssymb,amsopn,amscd}
\allowdisplaybreaks 
\usepackage[arrow,matrix]{xy}
\usepackage{cases}

\usepackage{fancyhdr}
\usepackage{datetime}
\newdateformat{generalisodate}{\THEYEAR.\twodigit\THEMONTH.\twodigit\THEDAY}
\newcommand{\todayiso}{\the\year\twodigit\month\twodigit\day}

\usepackage[pdftex,%
	bookmarks=true,%
	bookmarksopen=true,%
	bookmarksopenlevel=5,%
	colorlinks=true,%
	pdftitle=Arxiv\ \todayiso:\ Shearer's\ measure\ and\ stochastic\ domination\ of\ product\ measures,%
	pdfauthor=Temmel,%
]{hyperref}
\input{COMMON/latex/macro_general}
\newcommand{\Then}{\,\Rightarrow\,}

\newcommand{\Iff}{\,\Leftrightarrow\,}
\newcommand{\ForAll}{\forall\,}
\newcommand{\Exists}{\exists\,}

\newcommand{\FirstAlign}{\quad\,\,}

\newcommand{\NatNum}{\mathbb{N}}
\newcommand{\NatNumZero}{\NatNum_0} 
\newcommand{\IntNum}{\mathbb{Z}}

\newcommand{\RealNum}{\mathbb{R}}

\newcommand{\Set}[1]{{\{#1\}}}
\newcommand{\BigSet}[1]{{\left\{#1\right\}}}
\newcommand{\IntegerSet}[1]{{[#1]}}

\newcommand{\Cardinality}[1]{{|#1|}}
\newcommand{\PowerSet}[1]{{\mathcal{P}(#1)}}

\newcommand{\Indicator}[1]{\,\mathbb{I}_{#1}}
\newcommand{\Iverson}[1]{{[#1]}}

\DeclareMathOperator{\Support}{supp}

\newcommand{\Floor}[1]{\lfloor#1\rfloor}

\newcommand{\Graph}[1]{G(#1)}

\newcommand{\Neighbours}[1]{\mathcal{N}(#1)}
\newcommand{\NeighboursIncluded}[1]{\mathcal{N}_1(#1)}

\newcommand{\GraphMetricSymbol}{d}
\newcommand{\GraphMetric}[2]{\GraphMetricSymbol(#1,#2)}

\newcommand{\AdjacentTo}{{\,\backsim\,}}
\newcommand{\NotAdjacentTo}{{\,\not\backsim\,}}

\newcommand{\Iid}{\text{i.i.d. \hspace{-0.4ex}}}

\newcommand{\Proba}{\mathbb{P}}
\newcommand{\Expect}{\mathbb{E}}

\newcommand{\MeasureSpaceSymbol}{\mathcal{M}}

\newcommand{\ProbabilityMeasureSpace}[1]{\MeasureSpaceSymbol_1(#1)}

\newcommand{\Stochastically}[1]{\stackrel{st}{#1}}

\newcommand{\Tree}{\mathbb{T}}

\newcommand{\GrowthRateSymbol}{gr}

\newcommand{\GrowthUpperRate}[1]{{\operatorname{\overline{\GrowthRateSymbol}}(#1)}}

\newcommand{\RenderName}[1]{#1} 
\newcommand{\MakeName}[2]{
	\newcommand{#1}{\RenderName{#2}}
}
\MakeName{\NameAaronson}{Aaronson}
\MakeName{\NameAaronsonEtAl}{\NameAaronson{} et al.}
\MakeName{\NameAaronsonGilatKeaneDeValk}{\NameAaronson{}, Gilat, Keane \& de Valk}
\MakeName{\NameAntalPisztora}{Antal \& Pisztora}
\MakeName{\NameBernoulli}{Bernoulli}
\MakeName{\NameBalister}{Balister}
\MakeName{\NameBollobas}{Bollob\'{a}s}
\MakeName{\NameBalisterBollobas}{\NameBalister{} \& \NameBollobas{}}
\MakeName{\NameBoltzmann}{Boltzmann}
\MakeName{\NameCayley}{Cayley}
\MakeName{\NameDijkstra}{Dijkstra}
\MakeName{\NameDobrushin}{Dobrushin}
\MakeName{\NameDobrushinLangfordRuelle}{\NameDobrushin{},Langford \& Ruelle}
\MakeName{\NameDurrett}{Durrett}
\MakeName{\NameErdos}{Erd\"{o}s}
\MakeName{\NameErdosLovasz}{\NameErdos{} \& \NameLovasz{}}
\MakeName{\NameFaris}{Faris}
\MakeName{\NameFernandez}{Fern\'{a}ndez}
\MakeName{\NameFernandezProcacci}{\NameFernandez{} \& \NameProcacci{}}
\MakeName{\NameFolner}{F\o{}lner}
\MakeName{\NameGaltonWatson}{Galton-Watson}
\MakeName{\NameGibbs}{Gibbs}
\MakeName{\NameGuttman}{Guttman}
\MakeName{\NameHausdorff}{Hausdorff}
\MakeName{\NameHessenberg}{Hessenberg}
\MakeName{\NameIverson}{Iverson}
\MakeName{\NameKolmogorov}{Kolmogorov}
\MakeName{\NameKotecky}{Koteck\'{y}}
\MakeName{\NameKoteckyPreiss}{\NameKotecky{} \& \NamePreiss{}}
\MakeName{\NameLebesgue}{Lebesgue}
\MakeName{\NameLiggett}{Liggett}
\MakeName{\NameLiggettEtAl}{\NameLiggett{} et al.}
\MakeName{\NameLiggettSchonmannStacey}{\NameLiggett{}, \NameSchonmann{} \& \NameStacey{}}
\MakeName{\NameLovaszLocalLemma}{\NameLovasz{} Local Lemma}
\MakeName{\NameLovasz}{Lov\'{a}sz}
\MakeName{\NameLyons}{Lyons}
\MakeName{\NameMathieu}{Mathieu}
\MakeName{\NameMayer}{Mayer}
\MakeName{\NameMiracleSole}{Miracle-Sol\'{e}}
\MakeName{\NamePenrose}{Penrose}
\MakeName{\NamePeres}{Peres}
\MakeName{\NamePreiss}{Preiss}
\MakeName{\NameProcacci}{Procacci}
\MakeName{\NameRamsey}{Ramsey}
\MakeName{\NameSchonmann}{Schonmann}
\MakeName{\NameScott}{Scott}
\MakeName{\NameScottSokal}{\NameScott{} \& \NameSokal{}}
\MakeName{\NameShearer}{Shearer}
\MakeName{\NameShearersMeasure}{\NameShearer{}'s measure}
\MakeName{\NameSokal}{Sokal}
\MakeName{\NameStacey}{Stacey}
\MakeName{\NameTarski}{Tarski}
\MakeName{\NameTemmel}{Temmel}
\MakeName{\NameTodo}{Todo}
\MakeName{\NameUrsell}{Ursell}
\MakeName{\NameVanHove}{van Hove}
\MakeName{\NameWoess}{Woess}
\MakeName{\NameYoung}{Young}
\MakeName{\NameZorn}{Zorn}

\usepackage{comment}
\specialcomment{NoteToSelf}
	{\begin{paragraph}{Note to self:}}
	{{\hfill$\Diamond$}\end{paragraph}\vspace{1em}}

\excludecomment{NoteToSelf}
\usepackage{amsthm}

\theoremstyle{plain}
\newtheorem{Thm}{Theorem}
\newtheorem{Lem}[Thm]{Lemma}
\newtheorem{Prop}[Thm]{Proposition}

\newtheorem{Conj}[Thm]{Conjecture}

\theoremstyle{definition}
\newtheorem{Def}[Thm]{Definition}

\theoremstyle{definition}

\theoremstyle{remark}

\newtheorem*{Rem}{Remark}

\theoremstyle{definition}

\theoremstyle{plain}

\theoremstyle{plain} 

\usepackage{rotating}
\newcommand{\MxyProofStructureVerticalInclusion}{\text{%
	\begin{sideways}\begin{sideways}\begin{sideways}
		$\subseteq$
	\end{sideways}\end{sideways}\end{sideways}
}}

\newcommand{\MxyProofStructureVerticalInequality}{\text{%
	\begin{sideways}\begin{sideways}\begin{sideways}
		$\ge$
	\end{sideways}\end{sideways}\end{sideways}
}}

\xyoption{arrow}
\xyoption{matrix}

\newcommand{\LabelNonShearerImpliesNonDomination}{\text{(ND)}}
\newcommand{\LabelShearerImpliesUniformDomination}{\text{(UD)}}

\newcommand{\MxyDominationProofInclusions}[2]{
\xymatrix@C=#1@R=#2{
	*+{\PShearerSetInterior{G}}="n11"
		\ar@{}[r]|-{\stackrel{\LabelShearerImpliesUniformDomination}{\subseteq}} &
	*+{\PUniformDominationSet{\DependencyClassWeak{G}}}="n12"
		\ar@{}[r]|-{\subseteq} &
	*+{\PDominationSet{\DependencyClassWeak{G}}}="n13" &
	*+{}="n14"\\
	*+{}="n21" &
	*+{\PUniformDominationSet{\DependencyClassStrong{G}}}="n22"
		\ar@{}[r]|-{\subseteq}
		\ar@{}[u]|-{\MxyProofStructureVerticalInclusion} &
	*+{\PDominationSet{\DependencyClassStrong{G}}}="n23"
		\ar@{}[r]|-{\stackrel{\LabelNonShearerImpliesNonDomination}{\subseteq}}
		\ar@{}[u]|-{\MxyProofStructureVerticalInclusion}  &
	*+{\PShearerSetInterior{G}}="n24"
}}

\newcommand{\MxyDominationProofInequalities}[2]{
\xymatrix@C=#1@R=#2{
	*+{\PShearer{G}}="n11"
		\ar@{}[r]|-{\stackrel{\LabelShearerImpliesUniformDomination}{\ge}} &
	*+{\PUniformDomination{\DependencyClassWeak{G}}}="n12"
		\ar@{}[r]|-{\ge} &
	*+{\PDomination{\DependencyClassWeak{G}}}="n13" &
	*+{}="n14"\\
	*+{}="n21" &
	*+{\PUniformDomination{\DependencyClassStrong{G}}}="n22"
		\ar@{}[r]|-{\ge}
		\ar@{}[u]|-{\MxyProofStructureVerticalInequality} &
	*+{\PDomination{\DependencyClassStrong{G}}}="n23"
		\ar@{}[r]|-{\stackrel{\LabelNonShearerImpliesNonDomination}{\ge}}
		\ar@{}[u]|-{\MxyProofStructureVerticalInequality}  &
	*+{\PShearer{G}}="n24"
}}
\newcommand{\MxyDominationProofInequalitiesCenter}[2]{
\xymatrix@C=#1@R=#2{
	*+{\PDomination{\DependencyClassWeak{G}}}="n12"
		\ar@{}[r]|-{\le} &
	*+{\PUniformDomination{\DependencyClassWeak{G}}}="n13"\\
	*+{\PDomination{\DependencyClassStrong{G}}}="n22"
		\ar@{}[r]|-{\le}
		\ar@{}[u]|-{\MxyProofStructureVerticalInequality} &
	*+{\PUniformDomination{\DependencyClassStrong{G}}}="n23"
		\ar@{}[u]|-{\MxyProofStructureVerticalInequality}
}}

\newcommand{\AutomorphismGroup}[1]{\operatorname{Aut}(#1)}
\newcommand{\BernoulliProductField}[2]{{\Pi_{#1}^{#2}}}
\newcommand{\Configurations}[1]{{\mathcal{X}_{#1}}}
\newcommand{\Cylinder}[2]{{\Pi_{#1}^{-1}(#2)}}

\newcommand{\DependencyClassWeak}[1]{{\mathcal{C}_{#1}^{\text{weak}}}}
\newcommand{\DependencyClassWeakInvariant}[1]{{\mathcal{C}_{#1\text{-inv}}^{\text{weak}}}}

\newcommand{\DependencyClassStrong}[1]{{\mathcal{C}_{#1}^{\text{strong}}}}
\newcommand{\DependencyClassStrongInvariant}[1]{{\mathcal{C}_{#1\text{-inv}}^{\text{strong}}}}

\newcommand{\DominatedValue}[1]{{\sigma(#1)}}
\newcommand{\DominatedVectors}[1]{{\Sigma(#1)}}
\newcommand{\FuzzKZ}{{\IntNum_{(k)}}}
\newcommand{\Independent}{\text{ indep}}
\newcommand{\MonotoneFunctionsOn}[1]{{\operatorname{Mon}(#1)}}

\newcommand{\PowerFracTriple}[3]{{\frac{#1^{#1}#2^{#2}}{#3^{#3}}}}
\newcommand{\PowerFracDual}[2]{{\PowerFracTriple{#1}{\phantom{}}{#2}}}

\newcommand{\ShearerCriticalFunction}[1]{{\Xi_{#1}}}
\newcommand{\ShearerMeasure}[2]{{\mu_{#1,#2}}}
\newcommand{\ShearerOVOEP}[2]{\alpha_{#1}^{#2}}

\newcommand{\StrictlyGreater}{\gg}
\newcommand{\StrictlyLess}{\ll}

\newcommand{\PDomination}[1]{{p_{dom}^{#1}}}
\newcommand{\QDomination}[1]{{q_{dom}^{#1}}}
\newcommand{\PDominationSet}[1]{\mathcal{P}_{dom}^{#1}}

\newcommand{\PUniformDomination}[1]{{p_{udom}^{#1}}}
\newcommand{\QUniformDomination}[1]{{q_{udom}^{#1}}}
\newcommand{\PUniformDominationSet}[1]{\mathcal{P}_{udom}^{#1}}

\newcommand{\PShearer}[1]{{p_{sh}^{#1}}}
\newcommand{\QShearer}[1]{{q_{sh}^{#1}}}
\newcommand{\PShearerSet}[1]{\mathcal{P}_{sh}^{#1}}
\newcommand{\PShearerSetBoundary}[1]{\partial\mathcal{P}_{sh}^{#1}}
\newcommand{\PShearerSetInterior}[1]{\mathring{\mathcal{P}}_{sh}^{#1}}
\newcommand{\PShearerSetInteriorInvariant}[1]{\mathring{\mathcal{P}}_{sh}^{#1\text{-inv}}}

\title{\NameShearer{}'s measure and stochastic domination of product measures}
\author{
	Christoph Temmel
		\footnote{5030 Institut für Mathematische Strukturtheorie, Technische Universität Graz, Steyrergasse 30/III, 8010 Graz, Austria}
		\footnote{Email: temmel@math.tugraz.at}
}
\date{}
\begin{document}

\generalisodate
\maketitle

\begin{abstract}
Let $G:=(V,E)$ be a locally finite graph. Let $\vec{p}\in[0,1]^V$. We show that \NameShearersMeasure{}, introduced in the context of the \NameLovaszLocalLemma{}, with marginal distribution determined by $\vec{p}$, exists on $G$ iff every \NameBernoulli{} random field with the same marginals and dependency graph $G$ dominates stochastically a non-trivial \NameBernoulli{} product field. Additionally we derive a non-trivial uniform lower bound for the parameter vector of the dominated \NameBernoulli{} product field. This generalises previous results by \NameLiggettSchonmannStacey{} in the homogeneous case, in particular on the $k$-fuzz of $\IntNum$. Using the connection between \NameShearersMeasure{} and a hardcore lattice gas established by \NameScottSokal{}, we transfer bounds derived from cluster expansions of lattice gas partition functions to the stochastic domination problem.
\end{abstract}



\newcommand{\Keywords}{
\begin{paragraph}{Keywords:}
  stochastic domination,
  \NameLovaszLocalLemma{},
  product measure,
  \NameBernoulli{} random field,
  stochastic order,
  hardcore lattice gas.
\end{paragraph}
}


\newcommand{\Classification}{
\begin{paragraph}{MSC 2010:}
  60E15 (primary), 
  60G60, 
  82B20, 
  05D40. 
\end{paragraph}
}

\Keywords
\Classification

$\phantom{}$\vspace{1em}\hrule\vspace{1em}
This is an extended version of \cite{Temmel_domination}.
\vspace{1em}\hrule\vspace{1em}

\tableofcontents
\listoffigures

\section{Introduction}
\label{sec:introduction}

The question under which conditions a \NameBernoulli{} random field (short BRF) stochastically dominates a \NameBernoulli{} product field (short BPF) is of interest in probability and percolation theory. Knowledge of this kind allows the transfer of results from the independent case to more general settings. Of particular interest are BRFs with a dependency structure described by a graph $G$ and prescribed common marginal parameter $p$, as they often arise from rescaling arguments \cite{Grimmett_percolation_2nd}, dependent models \cite{BollobasRiordan_percolation} or particle systems \cite{Liggett_ips}. In this setting an interesting question is to find lower bounds on $p$ which guarantee stochastic domination for every such BRF.\\

This question has been investigated in the setting of boot-strap percolation \cite[section 2]{Andjel_characteristic} and super-critical \NameBernoulli{} percolation \cite[section 2]{AntalPisztora_chemical}. Finally \NameLiggettSchonmannStacey{} \cite{LiggettSchonmannStacey_domination} derived a generic lower bound for dependency graphs with uniformly bounded degree. Of particular interest is the $k$-fuzz of $\IntNum$ (short $\FuzzKZ$, that is the graph with vertex set $\IntNum$ and edges between all integers at distance less than or equal to $k$), which is the dependency graph of $k$-dependent BRFs on $\IntNum$. In this case they determined the minimal $p$, for which stochastic domination of a non-trivial BPF holds for each such BRF on $\FuzzKZ$. Even more, they showed, that in this case the parameter of the dominated BPF is uniformly bounded from below and nonzero for this minimal $p$ and made a conjecture about the size of the jump of the value of the parameter of the dominated BPF at this minimal $p$.\\

Their main tools have been a sufficient condition highly reminiscent of the \NameLovaszLocalLemma{} \cite{ErdosLovasz_hypergraphs} (short LLL, also known as the \NameDobrushin{} condition \cite{Dobrushin_perturbation} in statistical mechanics) and the explicit use of \NameShearersMeasure{} \cite{Shearer_problem} on $\FuzzKZ$ to construct a series of probability measures dominating only trivial BPFs. Recall that \NameShearersMeasure{} is the uniform minimal probability measure in the context of the LLL. It is also related to the grand canonical partition function of a lattice gas with both hard-core interaction and hard-core self-repulsion \cite{ScottSokal_repulsive,BissacotFernandezProccaciScoppola_improvement}.\\

Extending the work of \NameLiggettSchonmannStacey{} in a natural way we demonstrate, that the use of \NameShearersMeasure{} and the overall similarity between their proof and those concerning only \NameShearersMeasure{} is not coincidence, but part of a larger picture. We show that there is a non-trivial uniform lower bound on the parameter vector of the BPF dominated by a BRF with marginal parameter vector $\vec{p}$ and dependency graph $G$ iff \NameShearersMeasure{} with prescribed marginal parameter vector $\vec{p}$ exists on $G$.\\

After reparametrisation, the set of admissible vectors $\vec{p}$ is equivalent to the poly-disc of absolute and uniform convergence of the cluster expansion of the partition function of a hard-core lattice gas around fugacity $\vec{0}$ \cite{ScottSokal_repulsive,BissacotFernandezProccaciScoppola_improvement} allowing a high-temperature expansion \cite{Dobrushin_perturbation}. This opens the door to a reinterpretation of results from cluster expansion techniques \cite{GruberKunz_general,FernandezProcacci_cluster,BissacotFernandezProccaci_convergence} or tree equivalence techniques \cite[sections 6 \& 8]{ScottSokal_repulsive}, leading to improved estimates on admissible $\vec{p}$ for the domination problem. Possible future lines of research include the search for probabilistic interpretations of these combinatorial and analytic results.\\

The layout of this paper is as follows: we formulate the stochastic domination problem in section \ref{sec:setupAndProblemStatement} and give a short introduction to \NameShearersMeasure{} in section \ref{sec:shearerPrimer}. Section \ref{sec:mainResultsAndDiscussion} contains our new results, followed by examples of reinterpreted bounds in section \ref{sec:reinterpretationOfBounds}. Finally section \ref{sec:weakInvariantCase} deals with the weak invariant case and we refute the conjecture by \NameLiggettSchonmannStacey{} concerning the minimality of \NameShearersMeasure{} for the dominated parameter on $\FuzzKZ$ in section \ref{sec:asymptoticJumpSizeKFuzzZ}.

\section{Setup and problem statement}
\label{sec:setupAndProblemStatement}
Let $G:=(V,E)$ be a locally finite graph. Denote by $\Neighbours{v}$ the \emph{set of neighbours} of $v$ and by $\NeighboursIncluded{v}:=\Neighbours{v}\uplus\Set{v}$ the \emph{neighbourhood of $v$ including $v$ itself}. For every $W\subseteq V$ denote by $\Graph{W}$ the \emph{subgraph of $G$ induced by} $W$.\\

Vectors are indexed by $V$, i.e. $\vec{x}:=(x_v)_{v\in V}$. Multiplication of vectors acts coordinate-wise. We have the natural partial order $\le$ on real-valued vectors. Of particular importance is the notion of $\vec{x}\StrictlyLess\vec{y}$, which means that there is a strict inequality in all coordinates.  For $W\subseteq V$ let $\vec{x}_W:=(x_v)_{v\in W}$, where needed for disambiguation. We otherwise ignore superfluous coordinates. If we use a scalar $x$ in place of a vector $\vec{x}$ we mean to use $\vec{x}=x\vec{1}$ and call this the \emph{homogeneous setting}. We \emph{always assume the relation} $q=1-p$, also in vectorized form and when having corresponding subscripts. Denote by $\Configurations{V}:=\Set{0,1}^V$ the compact \emph{space of binary configurations} indexed by $V$. Equip $\Configurations{V}$ with the natural partial order induced by $\vec{x}\le\vec{y}$ (isomorph to the partial order induced by the subset relation in $\PowerSet{V}$).\\

A \emph{\NameBernoulli{} random field} (short BRF) $Y:=(Y_v)_{v\in V}$ on $G$ is a rv taking values in $\Configurations{V}$, seen as a collection of \NameBernoulli{} rvs $Y_v$ indexed by $V$. A \emph{\NameBernoulli{} product field} (short BPF) $X$ is a BRF where $(X_v)_{v\in V}$ is a collection of independent \NameBernoulli{} rvs. We write its law as $\BernoulliProductField{\vec{x}}{V}$, where $x_v:=\BernoulliProductField{\vec{x}}{V}(X_v=1)$.\\

A subset $A$ of the space $\Configurations{V}$ or the space $[0,1]^V$ is an \emph{up-set} iff
\begin{equation}\label{eq:upSet}
	\ForAll \vec{x}\in A,\vec{y}\in\Configurations{V}:\quad
	\vec{x}\le\vec{y} \Then \vec{y}\in A\,.
\end{equation}
Replacing $\le$ by $\ge$ in \eqref{eq:upSet} we define a \emph{down-set}.\\

We recall the definition of \emph{stochastic domination} \cite{Liggett_ips}. Let $Y$ and $Z$ be two BRFs on $G$. Denote by $\MonotoneFunctionsOn{V}$ the set of \emph{monotone continuous functions} from $\Configurations{V}$ to $\RealNum$, that is $\vec{s}\le\vec{t}$ implies $f(\vec{s})\le f(\vec{t})$. We say that $Y$ dominates $Z$ stochastically iff they respect monotonicity in expectation:
\begin{equation}\label{eq:stochasticDomination}
	Y\Stochastically{\ge}Z \Iff
	\Big(\ForAll f\in\MonotoneFunctionsOn{V}:
	\quad\Expect[f(Y)]\ge\Expect[f(Z)]\,\Big)\,.
\end{equation}
Equation \eqref{eq:stochasticDomination} actually refers to the laws of $Y$ and $Z$. We abuse notation and treat a BRF and its law as interchangeable. Stochastic domination is equivalent to the existence of a coupling of $Y$ and $Z$ with $\Proba(Y\ge Z)=1$ \cite{Strassen_existenceGivenMarginals}.\\

The \emph{set of all dominated \NameBernoulli{} parameter vectors} (short: set of dominated vectors) by a BRF Y is
\begin{subequations}\label{eq:baseDefinitionsInhomogeneous}
\begin{equation}\label{eq:dominated\NameBernoulli{}Vectors}
	\DominatedVectors{Y}
	:=\Set{\vec{c}: Y\Stochastically{\ge}\BernoulliProductField{\vec{c}}{V}}\,.
\end{equation}
It describes all the different BPFs minorating $Y$ stochastically. The set $\DominatedVectors{Y}$ is a closed down-set. The definition of dominated vector extends to a non-empty class $C$ of BRFs by
\begin{equation}\label{eq:dominatedClassVectors}
	\DominatedVectors{C}
	:=\bigcap_{Y\in C} \DominatedVectors{Y}\\
	=\Set{
		\vec{c}: \ForAll Y\in C:
		Y\Stochastically{\ge}\BernoulliProductField{\vec{c}}{V}
	}\,.
\end{equation}
For a class $C$ of BRFs denote by $C(\vec{p})$ the subclass consisting of BRFs with marginal parameter vector $\vec{p}$. We call a BPF with law $\BernoulliProductField{\vec{c}}{V}$, respectively the vector $\vec{c}$, \emph{non-trivial} iff $\vec{c}\StrictlyGreater0$. Our \emph{main question} is under which conditions all BRFs in a class $C$ dominate a non-trivial BPF. Even stronger, we ask whether they all dominate a common non-trivial BPF. Hence, given a class $C$, we investigate the \emph{set of parameter vectors guaranteeing non-trivial domination}
\begin{equation}\label{eq:dominatedParametersNonuniform}
	\PDominationSet{C}:=\BigSet{
		\vec{p}\in[0,1]^V:
		\ForAll Y\in C(\vec{p}):
		\Exists \vec{c}\StrictlyGreater\vec{0}:
		\vec{c}\in\DominatedVectors{Y}
	}
\end{equation}
and the \emph{set of parameter vectors guaranteeing uniform non-trivial domination}
\begin{equation}\label{eq:dominatedParametersUniform}
	\PUniformDominationSet{C}:=\BigSet{
		\vec{p}\in[0,1]^V:
		\Exists \vec{c}\StrictlyGreater\vec{0}:
		\vec{c}\in\DominatedVectors{C(\vec{p})}
	}\,.
\end{equation}
We have the obvious inclusion
\begin{equation}\label{eq:dominatedParametersInclusion}
	\PUniformDominationSet{C}\subseteq\PDominationSet{C}\,.
\end{equation}
\end{subequations}
The main contribution of this paper is the characterization and description of certain properties of the sets \eqref{eq:dominatedParametersUniform} and \eqref{eq:dominatedParametersNonuniform} for some classes of BRFs.\\

A first class of BRFs is the so-called \emph{weak dependency class} \cite[(1.1)]{LiggettSchonmannStacey_domination} with marginal parameter $\vec{p}$ on $G$:
\begin{equation}\label{eq:dependencyClassWeak}
	\DependencyClassWeak{G}(\vec{p}):=
		\Set{\text{BRF }Y: \ForAll v\in V:
		\Proba(Y_v=1|Y_{V\setminus\NeighboursIncluded{v}})\ge p_v}\,.
\end{equation}

\begin{NoteToSelf}
The class of \emph{lop-sided} probability measures \cite[page 70]{AlonSpencer_probmethod_third}, \cite[theorem 1.2]{ScottSokal_repulsive} is a subclass of $\DependencyClassWeak{G}(\vec{p})$.
\end{NoteToSelf}

\begin{NoteToSelf}
In the rhs of \eqref{eq:dependencyClassWeak} an $=p_v$ would be too restrictive. Nevertheless, $=p_v$ for all $v$ does not imply that $Y\in\DependencyClassStrong{G}(\vec{p})$ -- \cite[end of section $2$]{LiggettSchonmannStacey_domination} demonstrates this in the mixing example of finite \NameShearersMeasure{}s on $\IntNum$.
\end{NoteToSelf}

In this context $G$ is a \emph{weak dependency graph} of $Y$. We say that $G$ is a \emph{strong dependency graph} of a BRF $Y$ iff
\begin{equation}\label{eq:dependencyGraphStrong}
	\ForAll W_1,W_2\subset V:\quad
	\GraphMetric{W_1}{W_2}>1\Then Y_{W_1}\text{ is independent of }Y_{W_2}\,.
\end{equation}
In both cases, adding edges does not change $G$'s status as dependency graph of $Y$. It is possible that $Y$ has multiple minimal dependency graphs \cite[section 4.1]{ScottSokal_repulsive}. The second class is the so-called \emph{strong dependency class} \cite[section $0$]{LiggettSchonmannStacey_domination} with marginal parameter $\vec{p}$ on $G$:
\begin{equation}\label{eq:dependencyClassStrong}
	\DependencyClassStrong{G}(\vec{p}):=
		\left\{\text{BRF }Y:
			\begin{gathered}
				\ForAll v\in V:\quad\Proba(Y_v=1)=p_v\\
				G\text{ is a strong dependency graph of }Y
			\end{gathered}
		\right\}\,.
\end{equation}
In particular
\begin{equation}\label{eq:strongIsWeak}
	\DependencyClassStrong{G}(\vec{p})
	\subseteq\DependencyClassWeak{G}(\vec{p})\,.
\end{equation}
In all but some trivial cases the inclusion  \ref{eq:strongIsWeak} is strict (see after theorem \ref{thm:weakInvariantDominationShearerEquivalence}).
\section{A primer on \NameShearersMeasure{}}
\label{sec:shearerPrimer}
This section contains an introduction to and overview of \NameShearersMeasure{}. The following construction is due to \NameShearer{} \cite{Shearer_problem}. Let $G:=(V,E)$ be finite and $\vec{p}\in[0,1]^V$. Recall that an \emph{independent set of vertices} (in the graph theoretic sense) contains no adjacent vertices. Create a signed measure $\ShearerMeasure{G}{\vec{p}}$ on $\Configurations{V}$ with strong dependency graph $G$ by setting the marginals
\begin{subequations}\label{eq:shearersMeasure}
\begin{equation}\label{eq:shearerZeroMarginalDefinition}
	\ForAll W\subseteq V:\quad
	\ShearerMeasure{G}{\vec{p}}(Y_W=\vec{0})
	:=\begin{cases}
		\prod_{v\in W} q_v &W\text{ independent,}\\
		0 &W\text{ not independent.}
	\end{cases}
\end{equation}
Use the \emph{inclusion-exclusion principle} to complete $\ShearerMeasure{G}{\vec{p}}$:
\begin{equation}\label{eq:shearerInclusionExclusion}
	\ForAll W\subseteq V:\quad
	\ShearerMeasure{G}{\vec{p}}(Y_W=\vec{0},Y_{V\setminus W}=\vec{1})
	:=\sum_{\substack{W\subseteq T\subseteq V\\T\Independent}}
	(-1)^{\Cardinality{T}-\Cardinality{W}} \prod_{v\in T} q_v\,.
\end{equation}
\end{subequations}

\begin{NoteToSelf}
Do not write the definition \eqref{eq:shearerZeroMarginalDefinition} via explicit cylinder sets, as it would harm readability.
\end{NoteToSelf}

Define the \emph{critical function} of \NameShearer{}'s signed measure on $G$ by
\begin{equation}\label{eq:shearerCriticalFunction}
	\ShearerCriticalFunction{G}:\quad [0,1]^V\to\RealNum\qquad
	\vec{p}\mapsto \ShearerCriticalFunction{G}(\vec{p})
	:=\ShearerMeasure{G}{\vec{p}}(Y_V=\vec{1})
	=\sum_{\substack{T\subseteq V\\T\Independent}} \prod_{v\in T} (-q_v)\,.
\end{equation}
In graph theory \eqref{eq:shearerCriticalFunction} is also known as the \emph{independent set polynomial} of $G$ \cite{FisherSolow_dependence,HoedeLi_clique} and in lattice gas theory as the grand \emph{canonical partition function} at negative fugacity $-\vec{q}$ \cite[section 2]{ScottSokal_repulsive}. It satisfies a \emph{fundamental identity} (an instance of a \emph{deletion-contraction identity})
\begin{equation}\label{eq:shearerCriticalFunctionFundamentalIdentity}
	\ForAll v\in V, \vec{p}\in[0,1]^V:\quad
		\ShearerCriticalFunction{G}(\vec{p})
		=\ShearerCriticalFunction{\Graph{V\setminus\Set{v}}}(\vec{p})
		-q_v\,\ShearerCriticalFunction{\Graph{V\setminus\NeighboursIncluded{v}}}(\vec{p})\,,
\end{equation}
derived from \eqref{eq:shearerCriticalFunction} by discriminating between independent sets containing $v$ and those which do not.\\

The \emph{set of admissible parameters for \NameShearersMeasure{}} is
\begin{equation}\label{eq:shearerParameters}
	\begin{aligned}
	\PShearerSet{G}
	:=&\Set{
		\vec{p}\in[0,1]^V:\quad
		\ShearerMeasure{G}{\vec{p}}\text{ is a probability measure}
	}\\
	=&\Set{
		\vec{p}\in[0,1]^V:\quad\ForAll W\subseteq V:\quad
		\ShearerCriticalFunction{\Graph{W}}(\vec{p})\ge0 }\,.
	\end{aligned}
\end{equation}
The set $\PShearerSet{G}$ is closed, strictly decreasing when adding edges and an up-set \cite[proposition 2.15 (b)]{ScottSokal_repulsive}, hence connected. It always contains the vector $\vec{1}$ and, unless $E=\emptyset$, never the vector $\vec{0}$. Therefore it is a non-trivial subset of $[0,1]^V$ (see also section \ref{sec:reinterpretationOfBounds}). The function $\ShearerCriticalFunction{G}$ is strictly increasing on $\PShearerSet{G}$. It is convenient to subdivide $\PShearerSet{G}$ further into its \emph{boundary}
\begin{equation}\label{eq:shearerParametersBoundary}
	\PShearerSetBoundary{G}
	:=\Set{\vec{p}:
		\ShearerCriticalFunction{G}(\vec{p})=0
		\text{ and }
		\ShearerMeasure{G}{\vec{p}}\text{ is a probability measure}
	}
\end{equation}
and \emph{interior} (both seen as subsets of the space $[0,1]^V$)
\begin{equation}\label{eq:shearerParametersInterior}
	\begin{aligned}
	\PShearerSetInterior{G}
	:=\PShearerSet{G}\setminus\PShearerSetBoundary{G}
	&=\Set{\vec{p}:
		\ShearerCriticalFunction{G}(\vec{p})>0
		\text{ and }
		\ShearerMeasure{G}{\vec{p}}\text{ is a probability measure}
	}\\
	&=\Set{\vec{p}: \ShearerCriticalFunction{H}(\vec{p})>0\text{ for all subgraphs }H\text{ of }G}\,.
	\end{aligned}
\end{equation}

Finally we see that for $\vec{p}\in\PShearerSet{G}$ the probability measure $\ShearerMeasure{G}{\vec{p}}$
\begin{subequations}\label{eq:shearerCharacterizationProbability}
\begin{gather}
	\text{has dependency graph $G$,}
	\label{eq:shearerCharacterizationProbabilityDependencyGraph}\\
	\text{has marginal parameter $\vec{p}$, i.e $\ForAll v\in V: \ShearerMeasure{G}{\vec{p}}(Y_v=1)=p_v$,}
	\label{eq:shearerCharacterizationProbabilityMarginalParameter}\\
	\text{and forbids neighbouring $0$s, i.e. $\ForAll (v,w)\in E: \ShearerMeasure{G}{\vec{p}}(Y_v=Y_w=0)=0$.}
	\label{eq:shearerCharacterizationProbabilityMarginalNeighbouringZeros}
\end{gather}
\end{subequations}
Properties \eqref{eq:shearerCharacterizationProbabilityDependencyGraph} and \eqref{eq:shearerCharacterizationProbabilityMarginalParameter} are equivalent to $\ShearerMeasure{G}{\vec{p}}\in\DependencyClassStrong{G}(\vec{p})$. Every probability measure $\nu$ on $\Configurations{V}$ fulfilling \eqref{eq:shearerCharacterizationProbability} can be constructed by \eqref{eq:shearersMeasure} and thus coincides with $\ShearerMeasure{G}{\vec{p}}$. Hence \eqref{eq:shearerCharacterizationProbability} \emph{characterizes} $\ShearerMeasure{G}{\vec{p}}$.\\

The importance of \NameShearersMeasure{} is due to its \emph{uniform minimality} with respect to certain conditional probabilities:

\begin{Lem}[{\cite[theorem 1]{Shearer_problem}}]\label{lem:shearerMinimality}
Let $\vec{p}\in\PShearerSet{G}$ and $Z\in\DependencyClassWeak{G}(\vec{p})$. Then $\ForAll W\subseteq V$:
\begin{subequations}\label{eq:shearerMinimality}
\begin{equation}\label{eq:shearerMinimalityCriticalFunction}
	\Proba(Z_W=\vec{1})
	\ge\ShearerMeasure{G}{\vec{p}}(Y_W=\vec{1})
	=\ShearerCriticalFunction{\Graph{W}}(\vec{p})
	\ge 0
\end{equation}
and $\ForAll W\subseteq U\subseteq V$: if $\ShearerCriticalFunction{\Graph{W}}(\vec{p})>0$, then
\begin{equation}\label{eq:shearerMinimalityOVOEP}
	\Proba(Z_{U}=\vec{1}|Z_{W}=\vec{1})
	\ge \ShearerMeasure{G}{\vec{p}}(Y_{U}=\vec{1}|Y_{W}=\vec{1})
	= \frac{\ShearerCriticalFunction{\Graph{U}}(\vec{p})}%
	       {\ShearerCriticalFunction{\Graph{W}}(\vec{p})}
	\ge 0\,.
\end{equation}
\end{subequations}
\end{Lem}
It is the cost of isolating $0$s, that drives and is equivalent to the above minimality.\\

If $G$ is \emph{infinite} define
\begin{equation}\label{eq:shearerParametersInfinite}
	\PShearerSet{G}:=
		\bigcap_{E'\subseteq E,\Cardinality{E'}<\infty}
		\PShearerSet{(V,E')}
	\qquad\text{ and }\qquad
	\PShearerSetInterior{G}:=
		\bigcap_{E'\subseteq E,\Cardinality{E'}<\infty}
		\PShearerSetInterior{(V,E')}\,.
\end{equation}
This is well defined \cite[(8.4)]{ScottSokal_repulsive}. The set $\PShearerSetInterior{G}$ is not the interior of the closed set $\PShearerSet{G}$ (discussed in detail in \cite[theorem 8.1]{ScottSokal_repulsive}). For $\vec{p}\in\PShearerSet{G}$ the family of marginals $\Set{\ShearerMeasure{\Graph{W}}{p}:W\subsetneq V,W\text{ finite}}$ forms a consistent family à la \NameKolmogorov{} \cite[(36.1) \& (36.2)]{Billingsley_probability}. Hence \NameKolmogorov{}'s existence theorem \cite[theorem 36.2]{Billingsley_probability} establishes the existence of an extension of this family, which we call $\ShearerMeasure{G}{\vec{p}}$. The $\pi$-$\lambda$ theorem \cite[theorem 3.3]{Billingsley_probability} asserts the uniqueness of this extension. Furthermore $\ShearerMeasure{G}{\vec{p}}$ has all the properties listed in \eqref{eq:shearerCharacterizationProbability} on the infinite graph $G$. Conversely let $\nu$ be a probability measure having the properties \eqref{eq:shearerCharacterizationProbability}. Then all its finite marginals have them, too, and they coincide with \NameShearersMeasure{}. Hence by the uniqueness of the \NameKolmogorov{} extension $\nu$ coincides with $\ShearerMeasure{G}{\vec{p}}$ and \eqref{eq:shearerCharacterizationProbability} \emph{characterizes} $\ShearerMeasure{G}{\vec{p}}$ also on infinite graphs.

\section{Main results and discussion}
\label{sec:mainResultsAndDiscussion}
Our main result is

\begin{Thm}\label{thm:shearerDominationEquivalenceInhomogeneous}
For every locally finite graph $G$, we have
\begin{equation}\label{eq:shearerParametersEqualsDominationParameterSets}
	\PDominationSet{\DependencyClassWeak{G}}
	=\PUniformDominationSet{\DependencyClassWeak{G}}
	=\PDominationSet{\DependencyClassStrong{G}}
	=\PUniformDominationSet{\DependencyClassStrong{G}}
	=\PShearerSetInterior{G}\,.
\end{equation}
\end{Thm}

Its proof is in section \ref{sec:proof}. Theorem \ref{thm:shearerDominationEquivalenceInhomogeneous} consists of two a priori unrelated statements: The first one consists of the left three equalities in \eqref{eq:shearerParametersEqualsDominationParameterSets}: uniform and non-uniform domination of a non-trivial BPF are the same, and even taking the smaller class $\DependencyClassStrong{G}$ does not admit more $\vec{p}$. The second one is that these sets are equivalent to the set of parameters for which \NameShearersMeasure{} exists. The minimality of \NameShearersMeasure{} (see lemma \ref{lem:shearerMinimality}) lets us construct BRFs dominating only trivial BPFs for $\vec{p}\not\in\PShearerSetInterior{G}$ (see section \ref{sec:nondomination}) and clarifies the role \NameShearersMeasure{} played as a counterexample in the work of \NameLiggettSchonmannStacey{} \cite[section 2]{LiggettSchonmannStacey_domination}. Even more, this minimality implies an explicit lower bound for the non-trivial uniform dominated vector:

\begin{Thm}\label{thm:uniformlyDominatedVector}
For $\vec{p}\in\PShearerSetInterior{G}$, define the vector $\vec{c}$ component-wise by
\begin{subnumcases}{\label{eq:uniformlyDominatedVector}c_v:=}
	\label{eq:uniformlyDominatedVectorTrival}
	1
	&if $p_v=1$
	\\\label{eq:uniformlyDominatedVectorFinite}
	1-\left(1-\ShearerCriticalFunction{G_v}(\vec{p})\right)^{1/\Cardinality{V_v}}
	&if $p_v<1$ and $\Cardinality{V_v}<\infty$
	\\\label{eq:uniformlyDominatedVectorInfinite}
	q_v\min\Set{q_w:w\in\Neighbours{v}\cap V_v}
	&if $p_v<1$ and $\Cardinality{V_v}=\infty$,
\end{subnumcases}
where $V_v$ are the vertices of the connected component of $v$ in the subgraph of $G$ induced by all vertices $v$ with $p_v<1$. Then $\vec{0}\StrictlyLess\vec{c}\in\DominatedVectors{\DependencyClassWeak{G}(\vec{p})}$.
\end{Thm}

The proof of theorem \ref{thm:uniformlyDominatedVector} is in section \ref{sec:domination}. For infinite, connected $G$ we have a \emph{discontinuous transition} in $\vec{c}$ as $\vec{p}$ approaches the boundary of $\PShearerSetInterior{G}$ \eqref{eq:uniformlyDominatedVectorInfinite}, while in the finite case it is continuous \eqref{eq:uniformlyDominatedVectorFinite}. On the other hand there are classes of BRFs having a continuous transition also in the infinite case, for example the class of $2$-factors on $\IntNum$ \cite[theorem 3.0]{LiggettSchonmannStacey_domination}.\\

Our proof trades accuracy in capturing all of $\PShearerSetInterior{G}$ against accuracy in the lower bound for the parameter of the dominated BPF. Intuitively it is clear, that $\DominatedVectors{\DependencyClassWeak{G}(\vec{p})}$ should increase with $\vec{p}$ \eqref{eq:dominatedParameterMonotonicity}, but our explicit lower bound \eqref{eq:uniformlyDominatedVectorInfinite} decreases in $\vec{p}$. There is an explicit growing lower bound already shown by \NameLiggettSchonmannStacey{} \cite[corollary 1.4]{LiggettSchonmannStacey_domination}, although only on a restricted set of parameters \eqref{eq:LLLClassicalCondition}.\\

Equation \eqref{eq:shearerMinimalityCriticalFunction} does not imply, that $\ShearerMeasure{G}{\vec{p}}\Stochastically{\le}Y$ for all $Y\in\DependencyClassWeak{G}(\vec{p})$: for a finite $W\subsetneq V$ take $f:=1-\Indicator{\Set{\vec{0}}}\in\MonotoneFunctionsOn{W}$ and see that $\BernoulliProductField{\vec{p}}{W}\Stochastically{\not\ge}\ShearerMeasure{\Graph{W}}{\vec{p}}$. Furthermore $\DominatedVectors{\ShearerMeasure{G}{\vec{p}}}$ is neither minimal nor maximal (with respect to set inclusion) in the class $\DependencyClassWeak{G}(\vec{p})$. The maximal law is $\BernoulliProductField{\vec{p}}{W}$ itself, as $[\vec{0},\vec{p}]=\DominatedVectors{\BernoulliProductField{\vec{p}}{W}}$. We give a counterexample to the minimality of $\DominatedVectors{\ShearerMeasure{G}{\vec{p}}}$ in section \ref{sec:asymptoticJumpSizeKFuzzZ}.
\subsection{Reinterpretation of bounds}
\label{sec:reinterpretationOfBounds}
Theorem \ref{thm:shearerDominationEquivalenceInhomogeneous} allows the application of conditions for admissible $\vec{p}$ for $\PUniformDominationSet{\DependencyClassWeak{G}}$ to $\PShearerSetInterior{G}$ and vice-versa. Hence we can play questions about the existence of a BRF dominating only trivial BPFs or the existence of \NameShearersMeasure{} back and forth. In the following we list known necessary or sufficient conditions for $\vec{p}$ to lie in $\PShearerSetInterior{G}$, most of them previously unknown for the domination problem. We assume that $G$ contains no isolated vertices. The classical sufficient condition for the existence of \NameShearersMeasure{} has been established independently several times and is known as either the ``\NameLovaszLocalLemma{}'' \cite{ErdosLovasz_hypergraphs} in graph theory or the ``\NameDobrushin{} condition'' \cite[theorem 6.1]{Dobrushin_perturbation} in statistical mechanics:

\begin{Thm}[{version of \cite[(2.13)]{FernandezProcacci_cluster}}]
\label{thm:LLLClassical}
Let $\vec{p}\in[0,1]^V$. If there exists $\vec{s}\in]0,\infty[^V$ such that
\begin{equation}\label{eq:LLLClassicalCondition}
	\ForAll v\in V:\quad
	q_v\,\prod_{w\in\NeighboursIncluded{v}} (1+s_w)\le s_v\,,
\end{equation}
then $\vec{p}\in\PShearerSetInterior{G}$.
\end{Thm}

In the homogeneous case there has been again a parallel and independent improvement on theorem \ref{thm:LLLClassical} by \NameLiggettSchonmannStacey{} in probability theory and \NameScottSokal{} in statistical mechanics. Here $\PShearer{G}$ is identified with the endpoint of the interval $[\PShearer{G},1]$ corresponding to $\PShearerSetInterior{G}$.

\begin{Thm}[{\cite[theorem 1.3]{LiggettSchonmannStacey_domination}, \cite[corollary 5.7]{ScottSokal_repulsive}}]
\label{thm:LLLHomogeneousEscaping}
If $G$ is uniformly bounded with degree $D$, then
\begin{equation}\label{eq:LLLHomogeneousEscaping}
	\PShearer{G}\le 1-\PowerFracDual{(D-1)}{D}\,.
\end{equation}
\end{Thm}

This leads to the only two cases of infinite graphs where $\PShearer{G}$ is exactly known, namely the $D$-regular tree $\Tree_D$ with $\PShearer{\Tree_D}=1-\PowerFracDual{(D-1)}{D}$ and $\FuzzKZ$, the $k$-fuzz of $\IntNum$, with $\PShearer{\FuzzKZ}=1-\PowerFracDual{k}{(k+1)}$. The complementary inequality is \cite[before theorem 2]{Shearer_problem} and \cite[corollary 2.2]{LiggettSchonmannStacey_domination} for $\Tree_d$ and $\FuzzKZ$ respectively. In these cases explicit constructions of \NameShearersMeasure{} are possible. See for example the construction as a $(k+1)$-factor in the case of $\FuzzKZ$ \cite[section 4.2]{MathieuTemmel_kindependent}.\\

\NameFernandezProcacci{} derived another more recent and elaborate sufficient condition for a vector $\vec{p}$ to lie in $\PShearerSetInterior{G}$:

\begin{Thm}[{\cite[theorem 1]{FernandezProcacci_cluster}}]
\label{thm:clusterExpansion}
Let $\vec{p}\in[0,1]^V$. If there exists $\vec{s}\in]0,\infty[^V$, such that
\begin{equation}\label{eq:clusterExpansionCondition}
	\ForAll v\in V:\quad
	q_v\,\ShearerCriticalFunction{\Graph{\NeighboursIncluded{v}}}(-\vec{s})\le s_v\,,
\end{equation}
then $\vec{p}\in\PShearerSetInterior{G}$.
\end{Thm}

The minus in \eqref{eq:clusterExpansionCondition} stems from their cluster expansion technique and assures that $\ShearerCriticalFunction{\Graph{\NeighboursIncluded{v}}}(-\vec{s})\ge 1$, whence $q_v\le 1$. The condition takes into account the local structure of $G$, via the triangles in $\NeighboursIncluded{v}$. It thus improves upon the LLL, which only considers the degree of $v$.\\

We present an example of a necessary condition by \NameScottSokal{} in the homogeneous case. Define the \emph{upper growth rate} of a tree $\Tree$ rooted at $o$ by
\begin{equation}\label{eq:upperGrowthRate}
	\GrowthUpperRate{\Tree}:=\limsup_{n\to\infty} \Cardinality{V_n}^{1/n}\,,
\end{equation}
where $V_n$ are the vertices of $\Tree$ at distance $n$ from $o$. Then we have

\begin{Thm}[{\cite[proposition 8.3]{ScottSokal_repulsive}}]
\label{thm:prunedSAWBound}
Let $G$ be infinite. Then
\begin{equation}\label{eq:prunedSAWBound}
	\PShearer{G}\ge 1- \PowerFracDual{\GrowthUpperRate{\Tree}}{(\GrowthUpperRate{\Tree}+1)}\,.
\end{equation}
Here $\Tree$ is a particular pruned subtree of the SAW (self-avoiding-walk) tree of $G$ defined in \cite[section 6.2]{ScottSokal_repulsive}.
\end{Thm}

The pruned subtree $\Tree$ referred to above stems from a recursive expansion of the critical function via the fundamental identity \eqref{eq:shearerCriticalFunctionFundamentalIdentity} and the subsequent identification of this calculation with the one on $\Tree$. It is a subtree of the SAW tree of $G$, which not only avoids revisiting previously visited nodes, but also some of their neighbours. An example demonstrating this result is the following statement \cite[(8.53)]{ScottSokal_repulsive}:
\begin{equation}\label{eq:prunedSAWZdSimpleBound}
	\PShearer{\IntNum^d}\ge 1-\PowerFracDual{d}{(d+1)}\,.
\end{equation}
It follows from the fact that one can embed a regular rank $d$ rooted tree in the pruned SAW $\Tree$ of $\IntNum^d$, whence $d\le\GrowthUpperRate{\Tree}$. For the full details we refer the reader to \cite[sections 6 \& 8]{ScottSokal_repulsive}.

\section{Proofs}
\label{sec:proof}
We prove theorem \ref{thm:shearerDominationEquivalenceInhomogeneous} by showing all inclusions outlined in figure \ref{fig:structureOfInclusionsInShearerDominationEquivalenceInhomogeneous}. The four center inclusions follow straight from \eqref{eq:dominatedParametersInclusion} and \eqref{eq:strongIsWeak}. The core part are two inclusions marked \LabelShearerImpliesUniformDomination{} and \LabelNonShearerImpliesNonDomination{} in figure \ref{fig:structureOfInclusionsInShearerDominationEquivalenceInhomogeneous}. The second inclusion \LabelNonShearerImpliesNonDomination{} generalizes an idea of \NameLiggettSchonmannStacey{} in section \ref{sec:nondomination}. The key is the usage of \NameShearersMeasure{} on finite subgraphs $H$ for suitable $\vec{p}\in\PShearerSetBoundary{H}$ to create BRFs dominating only trivial BPFs. Our novel contribution is the inclusion \LabelShearerImpliesUniformDomination{}. It replaces the LLL style proof for restricted parameters employed in \cite[proposition 1.2]{LiggettSchonmannStacey_domination} by an optimal bound reminiscent of the optimal bound presented in \cite[section 5.3]{ScottSokal_repulsive}, using the fundamental identity \eqref{eq:shearerCriticalFunctionFundamentalIdentity} to full extent. After some preliminary work on \NameShearersMeasure{} in section \ref{sec:shearerOVOEP} we prove the inclusion \LabelShearerImpliesUniformDomination{} in section \ref{sec:domination}.

\begin{figure}[!htbp]
\begin{displaymath}
	\MxyDominationProofInclusions{0.5cm}{0.5cm}
\end{displaymath}
\caption{Inclusions in the proof of \eqref{eq:shearerParametersEqualsDominationParameterSets}. }
\label{fig:structureOfInclusionsInShearerDominationEquivalenceInhomogeneous}
\end{figure}

\subsection{Tools for stochastic domination}
\label{sec:toolsForStochasticDomination}
In this section we list useful statements related to stochastic domination between BRFs.

\begin{Lem}[{\cite[chapter II, page 79]{Liggett_ips}}]
\label{lem:dominationAndSubfields}
Let $Y,Z$ be two BRFs indexed by $V$, then
\begin{equation}\label{eq:dominationAndSubfields}
	Y\Stochastically{\ge}Z\quad\Iff\quad
	\left(
		\ForAll\text{ finite }W\subseteq V:
		Y_W\Stochastically{\ge}Z_W
	 \right)\,.
\end{equation}
\end{Lem}

\begin{NoteToSelf}
One way to prove lemma \ref{lem:dominationAndSubfields} is using the Lemma of Zorn on \NameKolmogorov{} consistent chains in the space of all coupling measures (in the context of the theorem by Liggett).
\end{NoteToSelf}

We build on the following technical result, inspired by \cite[lemma 1]{Russo_zeroone}.

\begin{Prop}\label{prop:russoInhomogeneousExtension}
If $Z:=\Set{Z_n}_{n\in\NatNum}$ is a BRF with
\begin{equation}\label{eq:russoInhomogeneousExtension}
	\ForAll n\in\NatNum,
	\vec{s}_{\IntegerSet{n}}\in\Configurations{{\IntegerSet{n}}}:\quad
	\Proba(Z_{n+1}=1|Z_{\IntegerSet{n}}=\vec{s}_{\IntegerSet{n}})\ge p_n\,,
\end{equation}
then $Z\Stochastically{\ge}\BernoulliProductField{\vec{p}}{\NatNum}$.
\end{Prop}

\begin{proof}
Essentially the same inductive proof as in \cite[lemma 1]{Russo_zeroone}.
\end{proof}

If $Y$ and $Z$ are two independent BRFs with marginal vectors $\vec{p}$ and $\vec{r}$, then we denote by
\begin{equation}\label{eq:vertexWiseMinimum}
	Y\land Z := (Y_v\land Z_v)_{v\in V}
\end{equation}
the \emph{vertex-wise minimum} with marginal vector $\vec{p}\,\vec{r}$. Coupling shows that for every two BRFs $Y$ and $Z$ we have
\begin{subequations}\label{eq:dominationVWMM}
\begin{equation}\label{eq:dominationVWMM_Chain}
	Y\land Z\Stochastically{\le} Y\,,
\end{equation}
and if $X$ is a third BRF independent of $(Y,Z)$ also
\begin{equation}\label{eq:dominationVWMM_Multiplication}
	Y\Stochastically{\ge}Z\Then
	(Y\land X)\Stochastically{\ge}(Z\land X)\,.
\end{equation}
\end{subequations}

\begin{Prop}\label{prop:dominatedParameterMonotonicity}
For each dependency class $C$ used in this paper and all $\vec{p}$ and $\vec{r}$, we have
\begin{equation}\label{eq:dominatedParameterMonotonicity}
	\DominatedVectors{C(\vec{p}\,\vec{r})}\subseteq\DominatedVectors{C(\vec{p})}\,.
\end{equation}
\end{Prop}

\begin{proof}
Let $\vec{c}\in\DominatedVectors{C(\vec{p}\,\vec{r})}$. Let $Y\in C(\vec{p})$ and $X$ be $\BernoulliProductField{\vec{r}}{V}$-distributed independently of $Y$. Using \eqref{eq:dominationVWMM} we get $\BernoulliProductField{\vec{c}}{V}\Stochastically{\le}Y\land X\Stochastically{\le}Y$, whence $\vec{c}\in\DominatedVectors{Y}$. As this holds for every $Y\in C(\vec{p})$ we have $\vec{c}\in\DominatedVectors{C(\vec{p})}$.
\end{proof}

\subsection{Nondomination}
\label{sec:nondomination}
In this section we prove inclusion \LabelNonShearerImpliesNonDomination{} from figure \ref{fig:structureOfInclusionsInShearerDominationEquivalenceInhomogeneous}, that is $\PDominationSet{\DependencyClassStrong{G}}\subseteq\PShearerSetInterior{G}$. The plan is as follows: in lemma \ref{lem:zeroProbabilityBelow} we recall a coupling involving \NameShearersMeasure{} on a finite graph $H$ \cite[proof of theorem 1]{Shearer_problem}, which creates a BRF dominating only trivial BPFs for every $\vec{p}\not\in\PShearerSetInterior{H}$. In proposition \ref{prop:nondomination} we generalize an approach used by \NameLiggettSchonmannStacey{} \cite[theorem $2.1$]{LiggettSchonmannStacey_domination} to arbitrary graphs and inhomogeneous parameters. For infinite $G$ and $\vec{p}\not\in\PShearerSetInterior{G}$, we find a suitable finite subgraph $H$ of $G$ on which to effectuate the above mentioned coupling and extend it with an independent BPF on the complement. The resulting BRF dominates only trivial BPFs.

\begin{Lem}[{\cite[proof of theorem $1$]{Shearer_problem}}]
\label{lem:zeroProbabilityBelow}
Let $G$ be finite. If $\vec{p}\not\in\PShearerSetInterior{G}$, then there exists a BRF $Z\in\DependencyClassStrong{G}(\vec{p})$ with $\Proba(Z_V=\vec{1})=0$.
\end{Lem}

\begin{proof}
As $\vec{p}\not\in\PShearerSetInterior{G}$ and $\vec{1}\in\PShearerSetInterior{G}$ the line segment $[\vec{p},\vec{1}]$ crosses $\PShearerSetBoundary{G}$ at the vector $\vec{r}$ (unique because $\PShearerSetInterior{G}$ is an up-set \cite[proposition 2.15 (b)]{ScottSokal_repulsive}). Let $\vec{x}$ be the solution of $\vec{p}=\vec{x}\,\vec{r}$. Let $Y$ be $\ShearerMeasure{G}{\vec{r}}$-distributed and $X$ be $\BernoulliProductField{\vec{x}}{V}$-distributed independently of $Y$. Set $Z:=Y\land X$. Then $Z\in\DependencyClassStrong{G}(\vec{p})$ and
\begin{equation*}
	\Proba(Z_V=\vec{1})=\Proba(X_V=\vec{1})\ShearerMeasure{G}{\vec{r}}(Y_V=\vec{1})=0\,.
\end{equation*}
\end{proof}

\begin{Prop}\label{prop:nondomination}
We have $\PDominationSet{\DependencyClassStrong{G}}\subseteq\PShearerSetInterior{G}$.
\end{Prop}

\begin{proof}
Let $\vec{p}\not\in\PShearerSetInterior{G}$. Then there exists a finite set $W\subseteq V$ with $\vec{p}_W\not\in\PShearerSetInterior{\Graph{W}}$. Using lemma \ref{lem:zeroProbabilityBelow} create a $Y_W\in\DependencyClassStrong{\Graph{W}}(\vec{p})$ with $\Proba(Y_W=\vec{1})=0$. Extend this to a $Y\in\DependencyClassStrong{G}(\vec{p})$ by letting $Y_{V\setminus W}$ be $\BernoulliProductField{\vec{p}_{V\setminus W}}{V\setminus W}$-distributed independently of $Y_W$. Suppose that $Y\Stochastically{\ge}X$, where $X$ is $\BernoulliProductField{\vec{x}}{V}$-distributed. Then lemma \ref{lem:dominationAndSubfields} implies that $Y_W\Stochastically{\ge}X_W$ and, using $f:=\Indicator{\Set{\vec{1}}}\in\MonotoneFunctionsOn{W}$, that
\begin{equation*}
	0
	=\Proba(Y_W=\vec{1})
	=\Expect[f(Y_W)]
	\ge\Expect[f(X_W)]
	=\Proba(X_W=\vec{1})
	=\prod_{v\in W} x_v
	\ge 0\,.
\end{equation*}
Hence there exists a $v\in W$ with $x_v=0$, whence $\vec{x}\not\StrictlyGreater\vec{0}$ and $\vec{p}\not\in\PDominationSet{\DependencyClassStrong{G}}$.
\end{proof}
\subsection{One vertex open extension probabilities}
\label{sec:shearerOVOEP}
In this section we reencode our knowledge of \NameShearersMeasure{} from the critical functions as ratios of critical functions, that is conditional probabilities of the form ``open on some vertices \textbar{} open on some other vertices''. These are exactly the ones \NameShearersMeasure{} is minimal for \eqref{eq:shearerMinimalityOVOEP}. This viewpoint admits a more succinct formulation of the fundamental identity \eqref{eq:shearerCriticalFunctionFundamentalIdentity} and bounds in proposition \ref{prop:shearerOVOEPBounds}. The notion of ``escaping'' pair introduced in this section is inspired by \cite[theorem 2]{Shearer_problem}. It allows us to push the mass of unwanted conditional events away. We obtain lower bounds on conditional events of the above form, which are independent of size of the condition, as long as the escape persists.\\

For finite $W\subseteq V$ with $v\not\in W$ and when $\ShearerCriticalFunction{\Graph{W}}(\vec{p})>0$ define the \emph{one vertex open extension probability} of $(W,v)$ by
\begin{equation}\label{eq:shearerOVOEP}
	\ShearerOVOEP{W}{v}(\vec{p})
	:=\ShearerMeasure{G}{\vec{p}}(Y_v=1|Y_W=\vec{1})\,.
\end{equation}
Reformulate the \emph{fundamental identity} \eqref{eq:shearerCriticalFunctionFundamentalIdentity} as
\begin{equation}\label{eq:shearerOVOEPFundamentalIdentity}
	\ShearerOVOEP{W}{v}(\vec{p})
	=1-\frac{q_v}{\prod_{i=1}^m \ShearerOVOEP{W\setminus\Set{w_i,\dotsc,w_m}}{w_i}(\vec{p})},
\end{equation}
where $W\cap\Neighbours{v}=:\Set{w_1,\dotsc,w_m}$.\\

\begin{Def}\label{def:shearerOVOEPEscaping}
Call the pair $(W,v)$, respectively $\ShearerOVOEP{W}{v}$, \emph{escaping} iff $\Neighbours{v}\setminus W\not=\emptyset$ and call every vertex $w\in\Neighbours{v}\setminus W$ an \emph{escape} of $(W,v)$.
\end{Def}

\begin{Prop}\label{prop:shearerOVOEPBounds}
Let $\vec{p}\in\PShearerSet{G}$, then
\begin{subequations}\label{eq:shearerOVOEPBounds}
\begin{equation}\label{eq:shearerOVOEPBoundsAbove}
	\ForAll (W,v):\quad
	\ShearerOVOEP{W}{v}(\vec{p})\le p_v
\end{equation}
and
\begin{equation}\label{eq:shearerOVOEPBoundsBelow}
	\ForAll (W,v), w\in\Neighbours{v}\setminus W:\quad
	q_w \le \ShearerOVOEP{W}{v}(\vec{p}).
\end{equation}
\end{subequations}
\end{Prop}

\begin{proof}
We use the fundamental identity \eqref{eq:shearerOVOEPFundamentalIdentity} to see that
\begin{equation*}
	\ShearerOVOEP{W}{v}(\vec{p})
	=1-\frac{q_v}{\prod \ShearerOVOEP{\star}{\star}(\vec{p})}
	\le 1-q_v = p_v\,.
\end{equation*}
Likewise, if $(W,v)$ is escaping with escape $w\in\Neighbours{v}\setminus W$, then \eqref{eq:shearerOVOEPFundamentalIdentity} yields
\begin{equation*}
  0
  \le \ShearerOVOEP{W\uplus\Set{v}}{w}(\vec{p})
  =  1-\frac{q_w}{\ShearerOVOEP{W}{v}(\vec{p}) \prod \ShearerOVOEP{\star}{\star}(\vec{p})}
 \le 1-\frac{q_w}{\ShearerOVOEP{W}{v}(\vec{p})}\,
\end{equation*}
hence $q_w\le\ShearerOVOEP{W}{v}(\vec{p})$.
\end{proof}

\begin{Prop}\label{prop:shearerOVOEPDecreasing}
Let $\vec{p}\in\PShearerSet{G}$. Then $\ShearerOVOEP{W}{v}(\vec{p})$ decreases, as $W$ increases.
\end{Prop}

\begin{proof}
If $p_v=0$, then $\ShearerOVOEP{W}{v}(\vec{p})=0$ for all $W$. If $p_w=0$ for $w\in W$ and $v$ connected to $w$ then $\ShearerOVOEP{W}{v}(\vec{p})$ is not defined. Hence for the remainder of this proof assume $\vec{0}\StrictlyLess\vec{p}$. We prove the statement by simultaneous induction for all $v$ over the cardinality of $W$. The base case is
\begin{equation*}
\ShearerOVOEP{\emptyset}{v}(\vec{p})=1-q_v
\begin{cases}
	\ge \frac{1-q_v-q_w}{1-q_w}=\ShearerOVOEP{\Set{w}}{v}(\vec{p}) &\text{if }v\AdjacentTo w\\
	= 1-q_v=\ShearerOVOEP{\Set{w}}{v}(\vec{p}) &\text{if }v\NotAdjacentTo w\, .
\end{cases}
\end{equation*}

For the induction step we add just one vertex $w$ to $W$ and set $U:=W\uplus\Set{w}$. Let $\Set{w_1,\dotsc,w_m}:=\Neighbours{v}\cap U$. First assume that $w\NotAdjacentTo v$. Using the fundamental identity \eqref{eq:shearerOVOEPFundamentalIdentity} we have
\begin{equation*}
\ShearerOVOEP{U}{v}(\vec{p})
= 1-\frac{q_v}{
	\prod_{i=1}^m
	\ShearerOVOEP{U\setminus\Set{w_i,\dotsc,w_m}}{w_i}(\vec{p})
	}
\le 1-\frac{q_v}{
	\prod_{i=1}^m
	\ShearerOVOEP{W\setminus\Set{w_i,\dotsc,w_m}}{w_i}(\vec{p})
	}
= \ShearerOVOEP{W}{v}(\vec{p})\,.
\end{equation*}
Secondly assume that $v\AdjacentTo w=w_m$. Hence
\begin{equation*}
\ShearerOVOEP{U}{v}(\vec{p})
= 1-\frac{q_v}{%
	\prod_{i=1}^m
	\ShearerOVOEP{U\setminus\Set{w_i,\dotsc,w_m}}{w_i}(\vec{p})
	}
\le 1-\frac{q_v}{
	\prod_{i=1}^{m-1}
	\ShearerOVOEP{W\setminus\Set{w_i,\dotsc,w_{m-1}}}{w_i}(\vec{p})
	}
= \ShearerOVOEP{W}{v}(\vec{p})\,.
\end{equation*}
\end{proof}
\subsection{Domination}
\label{sec:domination}
In this section we prove inclusion \LabelShearerImpliesUniformDomination{} from figure \ref{fig:structureOfInclusionsInShearerDominationEquivalenceInhomogeneous}, that is $\PShearerSetInterior{G}\subseteq\PUniformDominationSet{\DependencyClassWeak{G}}$.  We split the proof in two and deal with finite and infinite $G$ separately in proposition \ref{prop:shearerImpliesUniformDominationFinite} and \ref{prop:shearerImpliesUniformDominationInfinite}, respectively. Additionally \eqref{eq:shearerImpliesUniformDominationFinite} and \eqref{eq:shearerImpliesUniformDominationInfinite} combined yield a proof of \eqref{eq:uniformlyDominatedVector} from theorem \ref{thm:uniformlyDominatedVector}.\\

On a finite graph our approach is direct: proposition \ref{prop:shearerImpliesUniformDominationFinite} uses the minimality of $\ShearerMeasure{G}{\vec{p}}$ to construct a homogeneous nontrivial dominated vector $\vec{0}\StrictlyLess\vec{c}\in\DominatedVectors{\DependencyClassWeak{G}(\vec{p})}$. For an infinite graph the situation is more involved and we use a technique of \NameAntalPisztora{} \cite[pages 1040--1041]{AntalPisztora_chemical}: Suppose you have a $Y\in\DependencyClassWeak{G}(\vec{p})$ with $\vec{0}\StrictlyLess\vec{y}\in\DominatedVectors{Y}$. Let $X$ be $\BernoulliProductField{V}{\vec{x}}$ with $\vec{0}\StrictlyLess\vec{x}$ independently of $Y$ and set $Z:=X\land Y$. Then $\vec{0}\StrictlyLess\vec{x}\,\vec{y}\in\DominatedVectors{Z}\subseteq\DominatedVectors{Y}$, that is an independent non-trivial \Iid{} perturbation does not change the quality of $Y$'s domination behaviour.\\

Proposition \ref{prop:dominationConditionalMinoration} uses this perturbation to blame adjacent $0$ realizations of $Z$ on $X$ instead of $Y$, leading to the uniform technical minorization \eqref{eq:dominationConditionalMinoration}:
\begin{equation*}
	\Proba(Z_v=1|Z_W=\vec{s}_W)\ge q_v\ShearerOVOEP{W}{v}(\vec{p})\,,
\end{equation*}
connecting the domination problem with \NameShearersMeasure{}. Finally in proposition \ref{prop:shearerImpliesUniformDominationInfinite} we ensure to look at only escaping $(W,v)$s, hence getting rid of the $\ShearerOVOEP{W}{v}(\vec{p})$ term. This allows us to apply proposition \ref{prop:russoInhomogeneousExtension} and guarantee stochastic domination of a non-trivial BPF.

\begin{Prop}\label{prop:shearerImpliesUniformDominationFinite}
Let $G$ be finite and $\vec{p}\in\PShearerSetInterior{G}$. Let $X$ be $\BernoulliProductField{c}{V}$-distributed with
\begin{equation}\label{eq:shearerImpliesUniformDominationFinite}
	c:=1-\left(1-\ShearerCriticalFunction{G}(\vec{p})\right)^{1/\Cardinality{V}}>0\,.
\end{equation}
Then every $Y\in\DependencyClassWeak{G}(\vec{p})$ fulfils $Y\Stochastically{\ge}X$, hence $\vec{p}\in\PUniformDominationSet{\DependencyClassWeak{G}}$.
\end{Prop}

\begin{proof}
The choice of $\vec{p}$ implies that $\ShearerCriticalFunction{G}(\vec{p})>0$, therefore $c>0$, too. Let $f\in\MonotoneFunctionsOn{V}$ and $Y\in\DependencyClassWeak{G}(\vec{p})$. Then
\begin{align*}
	&\FirstAlign \Expect[f(X)]\\
	&= \sum_{\vec{s}\in\Configurations{V}} f(\vec{s})\,\Proba(X=\vec{s})\\
	&\le f(\vec{0})\,\Proba(X=\vec{0})+f(\vec{1})\,\Proba(X\not=\vec{0})
		&\text{monotonicity of }f\\
	&= f(\vec{0})(1-c)^{\Cardinality{V}}+f(\vec{1})[1-(1-c)^{\Cardinality{V}}]\\
	&= f(\vec{0})[1-\ShearerCriticalFunction{G}(\vec{p})]
	 + f(\vec{1})\,\ShearerCriticalFunction{G}(\vec{p})\\
	&\le f(\vec{0})\,\Proba(Y\not=\vec{1})+f(\vec{1})\,\Proba(Y=\vec{1})
		&\text{minimality of \NameShearersMeasure{} }
		       \eqref{eq:shearerMinimalityCriticalFunction}\\
	&\le \sum_{\vec{s}\in\Configurations{V}} f(\vec{s})\,\Proba(Y=\vec{s})
		&\text{monotonicity of }f\\
	&= \Expect[f(Y)]\,.
\end{align*}
Hence $X\Stochastically{\le}Y$. As $\vec{0}\StrictlyLess c\,\vec{1}$ we have $\vec{p}\in\PUniformDominationSet{\DependencyClassWeak{G}}$.
\end{proof}

\begin{Prop}\label{prop:shearerImpliesUniformDominationInfinite}
Let $G$ be infinite and connected. Let $\vec{1}\StrictlyGreater\vec{p}\in\PShearerSetInterior{G}$. Define the vector $\vec{c}$ by
\begin{equation}\label{eq:shearerImpliesUniformDominationInfinite}
	\ForAll v\in V:\quad
	c_v:=q_v\min\Set{q_w:w\in\Neighbours{v}}\,.
\end{equation}
Then $\vec{c}\StrictlyGreater\vec{0}$ and every $Y\in\DependencyClassWeak{G}(\vec{p})$ fulfils $Y\Stochastically{\ge}\BernoulliProductField{\vec{c}}{V}$, whence $\vec{p}\in\PUniformDominationSet{\DependencyClassWeak{G}}$.
\end{Prop}

\begin{Rem}
Proposition \ref{prop:shearerImpliesUniformDominationInfinite} motivated the definition of ``escaping'' pairs: it allows for non-trivial lower bounds for escaping $\ShearerOVOEP{W}{v}(\vec{p})$, in a correctly chosen ordering of a finite subgraph. Arbitrary $\ShearerOVOEP{W}{v}(\vec{p})$ defy control at the boundary of $\PShearerSetInterior{G}$.
\end{Rem}

\begin{proof}
We show, that $Y_W\Stochastically{\ge}\BernoulliProductField{\vec{c}_W}{W}$, for every finite $W\subsetneq V$. Admitting this momentarily, lemma \ref{lem:dominationAndSubfields} asserts that $Y\Stochastically{\ge}\BernoulliProductField{\vec{c}}{V}$. Conclude as $\vec{p}\StrictlyLess\vec{1}$ implies, that $\vec{c}\StrictlyGreater\vec{0}$.\\

Choose a finite $W\subsetneq V$ and let $\Cardinality{W}=:n$. As $G$ is connected and infinite, there is a vertex $v_n\in W$ which has a neighbour $w_n$ in $V\setminus W$. It follows, that $(W\setminus\Set{v_n},v_n)$ is escaping with escape $w_n\in\Neighbours{v_n}\setminus W$. Apply this argument recursively to $W\setminus\Set{v_n}$ and thus produce a total ordering $v_1\prec\dotsc\prec v_n$ of $W$, where, setting $W_i:=\Set{v_1,\dotsc,v_{i-1}}$, every $(W_i,v_i)$ is escaping with escape $w_i\in\Neighbours{v_i}\setminus W_i$.\\

Let $X$ be $\BernoulliProductField{\vec{q}}{V}$-distributed independently of $Y$. Set $Z:=Y\land X$. Then \eqref{eq:dominationConditionalMinoration} from proposition \ref{prop:dominationConditionalMinoration} and the minoration for escaping pairs \eqref{eq:shearerOVOEPBoundsBelow} combine to
\begin{equation*}
	\ForAll i\in\IntegerSet{n},
	\ForAll \vec{s}_{W_i}\in\Configurations{W_i}:\quad
		\Proba(Z_{v_i}=1|Z_{W_i}=\vec{s}_{W_i})
		\ge\ShearerOVOEP{W_i}{v_i}(\vec{p}) q_{v_i}
		\ge q_{w_i} q_{v_i}
		\ge c_{v_i}\,.
\end{equation*}
This is sufficient for proposition \ref{prop:russoInhomogeneousExtension} to construct a coupling with $Z_W\Stochastically{\ge}\BernoulliProductField{\vec{c}_W}{W}$. Apply \eqref{eq:dominationVWMM_Chain} to get
\begin{equation*}
	Y_W
	\Stochastically{\ge}Y_W\land X_W
	= Z_W
	\Stochastically{\ge}\BernoulliProductField{\vec{c}_W}{W}
\end{equation*}
and extend this to all of $V$ with the help of lemma \ref{lem:dominationAndSubfields}.
\end{proof}

\begin{Prop}\label{prop:dominationConditionalMinoration}
Let $\vec{1}\StrictlyGreater\vec{p}\in\PShearerSetInterior{G}$ and $Y\in\DependencyClassWeak{G}(\vec{p})$. Let $X$ be $\BernoulliProductField{\vec{q}}{V}$-distributed independently of $Y$ and set $Z:=X\land Y$. We claim that for all admissible $(W,v)$
\begin{equation}\label{eq:dominationConditionalMinoration}
	\ForAll \vec{s}_W\in\Configurations{W}:\quad
	\Proba(Z_v=1|Z_W=\vec{s}_W)
	\ge  q_v\ShearerOVOEP{W}{v}(\vec{p}).
\end{equation}
\end{Prop}

\begin{Rem}
This generalizes \cite[proposition 1.2]{LiggettSchonmannStacey_domination}, the core of \NameLiggettSchonmannStacey{}'s proof, in the following ways: we localize the parameters $\alpha$ and $r$ they used and assume no total ordering of the vertices yet. Furthermore $r_v=q_v$ follows from a conservative bound of the form
\begin{equation*}
	r_v
	:= 1 - \sup\Set{\ShearerOVOEP{W}{v}(\vec{p}): (W,v)\text{ escaping}}
	= 1-p_v = q_v\,,
\end{equation*}
where the $\sup$ is attained in $\ShearerOVOEP{\emptyset}{v}(\vec{p})=p_v$.
\end{Rem}

\begin{NoteToSelf}
Lopsided conditions are sufficient for the LLL but fail in the proof of proposition \ref{prop:dominationConditionalMinoration} in \eqref{eq:dominationConditionalMinorationProofClass}, as in general $\vec{s}_M\not=\vec{1}$.
\end{NoteToSelf}

\begin{proof}
Recall that $\vec{p}\in\PShearerSetInterior{G}$ implies that $\vec{p}\StrictlyGreater\vec{0}$. Whence $\vec{q}\StrictlyLess\vec{1}$ and \eqref{eq:dominationConditionalMinoration} is well defined because
\begin{equation*}
	\ForAll\text{ finite }W\subseteq V,\vec{s}_W\in\Configurations{W}:
	\quad\Proba(Z_W=\vec{s}_W)>0\,.
\end{equation*}

For every decomposition $N_0\uplus N_1:=\Neighbours{v}\cap W$ with $N_0=:\Set{u_1,\dotsc,u_l}$, $N_1=:\Set{w_1,\dotsc,w_m}$ and  $M:=W\setminus\Neighbours{v}$ the fundamental identity \eqref{eq:shearerOVOEPFundamentalIdentity} implies the inequality
\begin{equation}\label{eq:LSSExtensionDisturbanceFundamentalInequality}
	[1-\ShearerOVOEP{W}{v}(\vec{p})]
	\left(\prod_{j=1}^l p_{u_j}\right)
	\prod_{i=1}^m \ShearerOVOEP{M\uplus\Set{w_1,\dotsc,w_{i-1}}}{w_i}(\vec{p})
	\ge q_v\,,
\end{equation}
where $p_{u_j}\ge \ShearerOVOEP{M\uplus N_1\uplus\Set{u_1,\dotsc,u_{j-1}}}{u_j}(\vec{p})$ follows from \eqref{eq:shearerOVOEPBoundsAbove}.\\

We prove \eqref{eq:dominationConditionalMinoration} inductively over the cardinality of $W$. The induction base $W=\emptyset$ is easy as $\Proba(Z_v=1)=q_v\Proba(Y_v=1)\ge q_v p_v = q_v\ShearerOVOEP{\emptyset}{v}(\vec{p})$.
For the induction step fix $\vec{s}_W\in\Configurations{W}$ and the decomposition
\begin{equation*}
	N_0:=\Set{w\in W\cap\Neighbours{v}: s_w=0}=:\Set{u_1,\dotsc,u_l}
\end{equation*}
and
\begin{equation*}
	N_1:=\Set{w\in W\cap\Neighbours{v}: s_w=1}=:\Set{w_1,\dotsc,w_m}\,.
\end{equation*}
We write
\begin{subequations}\label{eq:dominationConditionalMinorationProof}
\begin{align}
	&\FirstAlign\Proba(Y_v=0|Z_W=\vec{s}_W)
		\notag\\
	&=\Proba(Y_v=0|Z_{N_0}=\vec{0},Z_{N_1}=\vec{1},Z_M=\vec{s}_M)
		\notag\\
	&=\frac%
		{\Proba(Y_v=0,Z_{N_0}=\vec{0},Z_{N_1}=\vec{1},Z_M=\vec{s}_M)}
		{\Proba(Z_{N_0}=\vec{0},Z_{N_1}=\vec{1},Z_M=\vec{s}_M)}
		\notag\\
	&\le\frac%
		{\Proba(Y_v=0,Z_M=\vec{s}_M)}
		{\Proba(X_{N_0}=\vec{0},Y_{N_1}=\vec{1},Z_M=\vec{s}_M)}
		\label{eq:dominationConditionalMinorationProofZDefinition}\\
	&=\frac%
		{\Proba(Y_v=0|Z_M=\vec{s}_M)\Proba(Z_M=\vec{s}_M)}
		{\Proba(X_{N_0}=\vec{0})\,\Proba(Y_{N_1}=\vec{1},Z_M=\vec{s}_M)}
		\label{eq:dominationConditionalMinorationProofYIndependence}\\
	&\le\frac%
		{q_v}
		{\Proba(X_{N_0}=\vec{0})\,\Proba(Y_{N_1}=\vec{1}|Z_M=\vec{s}_M)}
		\label{eq:dominationConditionalMinorationProofClass}\\
	&=\frac%
		{q_v}
		{\prod_{j=1}^l (1-q_{u_j}) \prod_{i=1}^m
		\Proba(Y_{w_i}=1|Y_{w_1}=\ldots=Y_{w_{i-1}}=1,Z_M=\vec{s}_M)}
		\notag\\
	&\le\frac%
		{q_v}
		{\prod_{j=1}^l p_{u_j} \prod_{i=1}^m
		\ShearerOVOEP{M\uplus\Set{w_1,\dotsc,w_{i-1}}}{w_i}(\vec{p})}
		\label{eq:dominationConditionalMinorationProofInduction}\\
	&\le 1-\ShearerOVOEP{W}{v}(\vec{p})\,.
		\label{eq:dominationConditionalMinorationProofInequality}
\end{align}
\end{subequations}
The key steps in \eqref{eq:dominationConditionalMinorationProof} are:
\begin{description}
	\item[\eqref{eq:dominationConditionalMinorationProofZDefinition}] increasing the numerator by dropping $Z_{N_0}=\vec{0}$ and $Z_{N_1}=\vec{1}$ while decreasing the denominator by using the definition of $Z$,
	\item[\eqref{eq:dominationConditionalMinorationProofClass}] as $\GraphMetric{v}{M}\ge 1$ and $Y\in\DependencyClassWeak{G}(\vec{p})$,
	\item[\eqref{eq:dominationConditionalMinorationProofYIndependence}] using the independence of $X_{N_0}$ from $(Y_{N_1},Z_M)$,
	\item[\eqref{eq:dominationConditionalMinorationProofInduction}] applying the induction hypothesis \eqref{eq:dominationConditionalMinoration} to the factors of the rhs product in the denominator, which have strictly smaller cardinality,
	\item[\eqref{eq:dominationConditionalMinorationProofInequality}] applying inequality \eqref{eq:LSSExtensionDisturbanceFundamentalInequality}.
\end{description}
Hence
\begin{equation*}
  \Proba(Z_v=1|Z_W=\vec{s}_W)
  \ge q_v \Proba(Y_v=1|Z_W=\vec{s}_W)
  \ge q_v \ShearerOVOEP{W}{v}(\vec{p})\,.
\end{equation*}
\end{proof}
\section{The weak invariant case}
\label{sec:weakInvariantCase}
In this section we extend our characterization to the case of BRFs with weak dependency graph, which are invariant under a group action. Let $\Gamma$ be a subgroup of $\AutomorphismGroup{G}$. A BRF $Y$ is \emph{$\Gamma$-invariant} iff
\begin{equation}\label{eq:gammaInvariantBRF}
	\ForAll \gamma\in\Gamma:\quad
	(\gamma Y):=(Y_{\gamma(v)})_{v\in V}\text{ has the same law as }Y\,.
\end{equation}
For a given $\Gamma$ and $\Gamma$-invariant $\vec{p}$ we denote by $\DependencyClassWeakInvariant{\Gamma}(\vec{p})$ the \emph{weak, $\Gamma$-invariant dependency class}, that is $\Gamma$-invariant BRFs with weak dependency graph $G$, and by $\DependencyClassStrongInvariant{\Gamma}(\vec{p})$ the corresponding strong version.\\

We call a pair $(G,\Gamma)$ \emph{partition exhaustive} iff there exists a sequence of partitions $(P_n)_{n\in\NatNum}$ of $V$ with $P_n:=(V_i^{(n)})_{i\in\NatNum}$, such that
\begin{subequations}\label{eq:partitionExhaustive}
\begin{gather}
	\label{eq:partitionExhaustiveRegularity}
	\ForAll n,i,j\in\NatNum:\quad
	\Graph{V_i^{(n)}}\text{ is isomorph to }\Graph{V_1^{(n)}}=:G_n\,,\\
	\label{eq:partitionExhaustiveOrbit}
	\ForAll n\in\NatNum:\quad
	\text{the orbit of $P_n$ under $\Gamma$ is finite,}\\
	\label{eq:partitionExhaustiveExhaustion}
	V_1^{(n)}\xrightarrow[n\to\infty]{}V,\text{ that is $(G_n)_{n\in\NatNum}$ exhausts $G$.}
\end{gather}
\end{subequations}
The kind of graphs we have in mind are regular infinite trees and tree-like graphs, $\IntNum^d$ and other regular lattices (triangular, hexagonal, \ldots). We think of the group $\Gamma$ to be generated by some of the natural shifts and rotations of the graph. An example are increasing regular rectangular decompositions of $\IntNum^d$ together with translations of $\IntNum^d$.

\begin{Thm}\label{thm:weakInvariantDominationShearerEquivalence}
Let $(G,\Gamma)$ be partition exhaustive. Then
\begin{equation}\label{eq:weakInvariantDominationShearerEquivalence}
	\PUniformDominationSet{\DependencyClassWeakInvariant{\Gamma}}
	=\PDominationSet{\DependencyClassWeakInvariant{\Gamma}}
	=\PShearerSetInteriorInvariant{\Gamma}
	:=\Set{\vec{p}\in\PShearerSetInterior{G}:
		\quad\vec{p}\text{ is $\Gamma$-invariant}
	}\,.
\end{equation}
\end{Thm}

\begin{Rem}
It follows from \eqref{eq:partitionExhaustive} that $\Gamma$ acts quasi-transitively on $G$. Hence $\PShearerSetInteriorInvariant{\Gamma}$ can be seen as a subset of a finite-dimensional space.\\

The mixing in \eqref{eq:nondominatingInvariantWeakMixing} destroys strong independence even in simple cases like $G=\IntNum$ and $\Gamma$ the group of translations of $\IntNum$ \cite[end of section $2$]{LiggettSchonmannStacey_domination}. The easiest way to see this is to let $G:=(\Set{v,w},\emptyset)$, $X^{(1)},X^{(2)}\in\DependencyClassStrong{G}(\vec{p})$ and $Y$ be Bernoulli($\frac{1}{2}$)-distributed, all independent of each other. Define $Z:=X^{(Y)}$ and ask if $\Proba(Z_v=Z_w=1)=\Proba(Z_v=1)\Proba(Z_w=1)$. This fails for most choices of $\vec{p}$. Calculations on slightly more complex graphs as $G=(\Set{u,v,w},\Set{(u,v)})$ show, that $Z$ from \eqref{eq:nondominatingInvariantWeakMixing} has no strong dependency graph. Thus the present approach, inspired by \cite[page 89]{LiggettSchonmannStacey_domination}, does not allow to characterize $\PUniformDominationSet{\DependencyClassStrongInvariant{\Gamma}}$ and $\PDominationSet{\DependencyClassStrongInvariant{\Gamma}}$.
\end{Rem}

\begin{proof}
As $\DependencyClassWeakInvariant{\Gamma}(\vec{p})$ is a subclass of $\DependencyClassWeak{G}(\vec{p})$ theorem \eqref{thm:shearerDominationEquivalenceInhomogeneous} implies, that $\PShearerSetInteriorInvariant{\Gamma}\subseteq\PUniformDominationSet{\DependencyClassWeakInvariant{\Gamma}}\subseteq\PDominationSet{\DependencyClassWeakInvariant{\Gamma}}$. We show $\PDominationSet{\DependencyClassWeakInvariant{\Gamma}}\subseteq\PShearerSetInteriorInvariant{\Gamma}$ by constructing a counterexample. If $\vec{p}\not\in\PShearerSetInteriorInvariant{\Gamma}$, then by \eqref{eq:partitionExhaustiveExhaustion} there exists a $n\in\NatNum$, such that $\vec{p}\not\in\PShearerSetInterior{G_n,\Gamma}$ (the intersection of the projections of $\Gamma$-invariant parameters on $G$ with $\PShearerSetInterior{G_n}$). Let $P:=P_n$ and let $(P^{(1)},\dotsc,P^{(k)})$ be its finite orbit under the action of $\Gamma$ \eqref{eq:partitionExhaustiveOrbit}. By \eqref{eq:partitionExhaustiveRegularity} each class $V_{(i,j)}\in P^{(j)}$ has a graph $\Graph{V_{(i,j)}}$ isomorph to $G_n$. Use lemma \ref{lem:zeroProbabilityBelow} to construct \Iid{} BPFs $Z^{(i,j)}\in\DependencyClassStrong{G_n}(\vec{p})$ with $\Proba(Z^{(i,j)}=\vec{1})=0$. For $j\in\IntegerSet{k}$, collate the $Z_{(i,j)}$ to a BPF $Z^{(j)}$. This works, as $P^{(j)}$ is a partition of $G$. By definition $Z^{(j)}\in\DependencyClassStrongInvariant{\Gamma}(\vec{p})$. Finally let $U$ be Uniform($\IntegerSet{k}$)-distributed and independent of everything else. Define the final BPF $Z$ by
\begin{equation}\label{eq:nondominatingInvariantWeakMixing}
	Z := \sum_{j=1}^k \Iverson{U=j} Z^{(j)}\,.
\end{equation}
We claim that $Z\in\DependencyClassWeakInvariant{\Gamma}(\vec{p})$. The mixing in \eqref{eq:nondominatingInvariantWeakMixing} keeps $Z\in\DependencyClassWeak{G}(\vec{p})$. To see its $\Gamma$-invariance, let $\gamma\in\Gamma$. The automorphism $\gamma$ acts injectively on $(P^{(1)},\dotsc,P^{(k)})$ and thus also on $\IntegerSet{k}$. Therefore, using the fact that $U$ is uniform and everything is constructed independently, we have
\begin{equation*}
	\gamma Z
	= \sum_{j=1}^k \Iverson{U=j} \gamma Z^{(j)}
	= \sum_{j=1}^k \Iverson{U=\gamma^{-1} j} Z^{(j)}
	= \sum_{j=1}^k \Iverson{U=j} Z^{(j)}
	= Z\,.
\end{equation*}
\end{proof}
\section{The asymptotic size of the jump on \texorpdfstring{$\FuzzKZ$}{the k-fuzz of the integers}}
\label{sec:asymptoticJumpSizeKFuzzZ}
\NameLiggettSchonmannStacey{} formulated the following conjecture about the size of the jump at the critical value on $\FuzzKZ$, the $k$-fuzz of $\IntNum$:

\begin{Conj}[{\cite[after corollary 2.2]{LiggettSchonmannStacey_domination}}]
\label{conj:LSSConjecture}
\begin{equation}\label{eq:LSSConjecture}
	\ForAll k\in\NatNumZero:\quad
	\DominatedValue{
		\DependencyClassWeak{\FuzzKZ}
		(\PUniformDomination{\DependencyClassWeak{\FuzzKZ}})
	}
	=\frac{k}{k+1}\,.
\end{equation}
\end{Conj}

We think that \NameLiggettSchonmannStacey{} were led by the intuition, that the extra randomness used in obtaining the above lower bound (see the $Y$ in \cite[proposition 1.2]{LiggettSchonmannStacey_domination} or the $X$ in the proof of proposition \ref{prop:dominationConditionalMinoration}) can be ignored in a suitable transitive setting. This would, in general, yield $\DominatedValue{\ShearerMeasure{G}{\PShearer{G}}}=\DominatedValue{\DependencyClassWeak{G}(\PUniformDomination{\DependencyClassWeak{G}})}$, and, in the particular case of $\FuzzKZ$, $\DominatedValue{\ShearerMeasure{\FuzzKZ}{\PShearer{\FuzzKZ}}}=\frac{k}{k+1}$ \cite[section 4.2]{MathieuTemmel_kindependent}, with $\PShearer{\FuzzKZ}=1-\PowerFracDual{k}{(k+1)}$ \cite[section 4.2]{MathieuTemmel_kindependent}.\\

Proposition \ref{prop:KFuzzZAsympotics} shows, that asymptotically $\DominatedValue{\DependencyClassWeak{\FuzzKZ}(\PShearer{\FuzzKZ})}$ is much closer to the lower bound of $\frac{k}{(k+1)^2}$ from \cite[corollary 2.5]{LiggettSchonmannStacey_domination}. This is caused by the increasing range of dependence, as $k\to\infty$, which allows for extreme correlations on the same order as the extra randomness used to decorrelate them.

\begin{Prop}\label{prop:KFuzzZAsympotics}
For $k\in\NatNum$ and $\Gamma_k$ the translations of $\IntNum$, let $C_k$ be either $\DependencyClassStrong{\FuzzKZ}$ or $\DependencyClassWeakInvariant{\Gamma_k}$. We have
\begin{equation}\label{eq:KFuzzZAsympotics}
	\ForAll\varepsilon>0:\Exists K(\varepsilon): \ForAll k\ge K:\quad
	\DominatedValue{C_k(\PShearer{\FuzzKZ})}
	\le\frac{1+(1+\varepsilon)\ln(k+1)}{k+1}\,.
\end{equation}
\end{Prop}

\begin{proof}
Let $\NeighboursIncluded{0}^{+}:=\Set{0,\ldots,k}$ be the non-negative closed half-ball of radius $k$ centred at $0$. Define a BRF $Y$ on $\IntNum$ by setting $\Proba(Y_{\NeighboursIncluded{0}^{+}}=\vec{1}):=\PDomination{\FuzzKZ}$, $\Proba(Y_{\NeighboursIncluded{0}^{+}}=\vec{0}):=\QDomination{\FuzzKZ}$ and letting $Y_{\IntNum\setminus\NeighboursIncluded{0}^{+}}$ be $\BernoulliProductField{\PDomination{\FuzzKZ}}{\IntNum\setminus\NeighboursIncluded{0}^{+}}$-distributed independently of $Y_{\NeighboursIncluded{0}^{+}}$.
As $Y\in\DependencyClassStrong{\FuzzKZ}(\PDomination{\FuzzKZ})$, \cite[corollary 2.5]{LiggettSchonmannStacey_domination} applies and $Y\Stochastically{\ge}X$, where $X$ is $\BernoulliProductField{\sigma}{\IntNum}$-distributed with $\sigma\in[\frac{k}{(k+1)^2},\frac{k}{k+1}]$. Lemma \ref{lem:dominationAndSubfields} implies $X_{\NeighboursIncluded{\vec{0}}^{+}}\Stochastically{\le}Y_{\NeighboursIncluded{\vec{0}}^{+}}$ and in particular the inequality
\begin{equation*}
	1-(1-\sigma)^{(k+1)}
	= \Proba(X_{\NeighboursIncluded{\vec{0}}^{+}}\not=\vec{0})
	= \Proba(Y_{\NeighboursIncluded{\vec{0}}^{+}}\not=\vec{0})
	= 1-\frac{k^k}{(k+1)^{(k+1)}}\,.
\end{equation*}
Rewrite it into
\begin{align*}
	\sigma
	&\le 1-\frac{k^{\frac{k}{k+1}}}{k+1}\\
	&=\frac{1}{k+1}
		+\frac{k}{k+1}
		(1-k^{-\frac{1}{k+1}})\\
	&\le \frac{1}{k+1}
		+(1-(k+1)^{-\frac{1}{k+1}})\ .
\end{align*}
For every $\varepsilon>0$ and $z$ close enough to $0$ we know that $1-e^{-z}\le(1+\varepsilon)z$. The statement for $\DependencyClassStrong{\FuzzKZ}$ follows from $z_k:=\frac{\ln(k+1)}{k+1}\xrightarrow[k\to\infty]{}0$. The result for $\DependencyClassWeakInvariant{\Gamma_k}$ follows from a mixing construction similar to \eqref{eq:nondominatingInvariantWeakMixing}.
\end{proof}
\section*{Acknowledgements}
\label{sec:acknowledgements}
I want to thank Yuval Peres and Rick Durrett for pointing out \cite{LiggettSchonmannStacey_domination} to me and Pierre Mathieu for listening patiently to my numerous attempts at understanding and solving this problem. This work has been partly done during a series of stays at the LATP, Aix-Marseille Université, financially supported by grants A3-16.M-93/2009-1 and A3-16.M-93/2009-2 from the Land Steiermark and by the Austrian Science Fund (FWF), project W1230-N13. I am also indebted to the anonymous referees for their constructive comments.
\bibliography{COMMON/references}
\section{Additional Material}
\label{sec:additional}
\subsection{Intrinsic coupling and domination of \NameShearersMeasure{}}
\label{sec:shearerIntrinsics}
In this section take a look at the parameters of the BPF dominated by \NameShearersMeasure{}. We specialize proposition \ref{prop:dominationConditionalMinoration} in proposition \ref{prop:intrinsicConditionalMinoration} and find that we do not need an auxiliary BPF. Therefore a natural vector in the set $\DominatedVectors{\ShearerMeasure{G}{\vec{p}}}$ is described by the the one-vertex open extensions probabilities in proposition \ref{prop:intrinsicDominationMinoration}. We only deal with connected graphs, as the results factorize over connected components.\\

The vertex-wise $\max$ operation is defined analogously to the vertex-wise minimum \eqref{eq:vertexWiseMinimum}. It erases $0$s in realizations, thinning out independent sets of $0$s. Hence it conserves \NameShearersMeasure{}. Formally, let $Y$ be $\ShearerMeasure{G}{\vec{p}}$-distribued and $X$ be $\BernoulliProductField{V}{\vec{c}}$-distributed independently of $Y$. Then $Y\lor X$ is $\ShearerMeasure{\vec{p}+\vec{c}-\vec{c}\,\vec{p}}{G}$-distributed. This is a \emph{coupling} between $\ShearerMeasure{\vec{p}}{V}$ and $\ShearerMeasure{\vec{p}+\vec{c}-\vec{c}\,\vec{p}}{G}$ and implies that
\begin{equation}\label{eq:intrinsicCoupling}
	\ShearerMeasure{\vec{p}}{V}
	\Stochastically{\le}
	\ShearerMeasure{\vec{p}+\vec{c}-\vec{c}\,\vec{p}}{G}\,,
\end{equation}
with equality iff $X=0$. Furthermore
\begin{equation*}
	\ForAll (W,v):\quad
	\ShearerOVOEP{W}{v}(\vec{p}+\vec{c}-\vec{c}\,\vec{p})
	= \ShearerOVOEP{W}{v}(\vec{p})+c_v\left[1-\ShearerOVOEP{W}{v}(\vec{p})\right]
	\ge \ShearerOVOEP{W}{v}(\vec{p})\,.
\end{equation*}
This implies the monotonicity of $\ShearerCriticalFunction{G}$ on $\PShearerSet{G}$, the fact that $\PShearerSet{G}$ and $\PShearerSetInterior{G}$ are up-sets and the monotonicity of $\vec{x}$ from \eqref{eq:intrinsicVector} in $\vec{p}$ with $\lim_{\vec{p}\to\vec{1}} \vec{x} = \vec{1}$.

\begin{Prop}\label{prop:intrinsicConditionalMinoration}
Let $G:=(V,E)$ be connected, $\vec{p}\in\PShearerSetInterior{G}$ and $Y$ be $\ShearerMeasure{G}{\vec{p}}$-distributed. We claim that for all admissible pairs $(W,v)$
\begin{equation}\label{eq:intrinsicConditionalMinoration}
	\ForAll\vec{s}_W\in\Configurations{W}:\quad
	\ShearerMeasure{G}{\vec{p}}(Y_v=1|Y_W=\vec{s}_W)
	\ge\ShearerOVOEP{W}{v}(\vec{p})
	>0\,.
\end{equation}
\end{Prop}

\begin{proof}
The fact that $\vec{p}\in\PShearerSetInterior{G}$ implies that all admissible $\ShearerOVOEP{W}{v}(\vec{p})$ are well defined and non-zero.\\

We prove \eqref{eq:intrinsicConditionalMinoration} inductively over the cardinality of $W$. The induction base for $W=\emptyset$ is $\ShearerMeasure{G}{\vec{p}}(Y_v=1)=p_v=\ShearerOVOEP{\emptyset}{v}(\vec{p})$. For the induction step let $M:=W\setminus\Neighbours{v}$ and $N:=W\cap\Neighbours{v}$. Let $\vec{s}_W\in\Configurations{W}$ and assume that $\ShearerMeasure{G}{\vec{p}}(Y_W=\vec{s}_W)>0$. The first case is $\vec{s}_N\not=\vec{1}$, whereby
\begin{equation*}
\ShearerMeasure{G}{\vec{p}}(Y_v=0|Y_W=\vec{s}_W)
=\frac%
	  {\ShearerMeasure{G}{\vec{p}}(Y_v=0,Y_N\not=\vec{1},Y_M=\vec{s}_M)}
	  {\ShearerMeasure{G}{\vec{p}}(Y_Y=\vec{s}_W)}
=0\,,
\end{equation*}
as there are neighbouring zeros in $(Y_v,Y_N)$. The second case is $\vec{s}_N=\vec{1}$. Let $\Set{w_1,\dotsc,w_m}:=N$.
Use the fundamental identity \eqref{eq:shearerOVOEPFundamentalIdentity} to get
\begin{align*}
	&\FirstAlign\ShearerMeasure{G}{\vec{p}}(Y_v=0|Y_W=\vec{s}_W)\\
	&=\frac%
	  {\ShearerMeasure{G}{\vec{p}}(Y_v=0,Y_N=\vec{1},Y_M=\vec{s}_M)}
	  {\ShearerMeasure{G}{\vec{p}}(Y_N=\vec{1},Y_M=\vec{s}_M)}\\
	&=\frac%
	  {\ShearerMeasure{G}{\vec{p}}(Y_v=0)\ShearerMeasure{G}{\vec{p}}(Y_M=\vec{s}_M)}
	  {\ShearerMeasure{G}{\vec{p}}(Y_N=\vec{1},Y_M=\vec{s}_M)}\\
	&=\frac{q_v}
	  {\ShearerMeasure{G}{\vec{p}}(Y_N=\vec{1}|Y_M=\vec{s}_M)}\\
	&=\frac{q_v}
	  {\prod_{i=1}^m \ShearerMeasure{G}{\vec{p}}
		  (Y_{w_i}=1|Y_{\Set{w_1,\dotsc,w_{i-1}}}=\vec{1}, Y_M=\vec{s}_M)}\\
	&\le\frac{q_v}
	  {\prod_{i=1}^m \ShearerOVOEP{M\uplus\Set{w_1,\dotsc,w_{i-1}}}{w_i}(\vec{p})}\\
	&= 1 - \ShearerOVOEP{W}{v}(\vec{p})\,.
\end{align*}
\end{proof}

\begin{Prop}\label{prop:intrinsicDominationMinoration}
Let $G$ be infinite and connected. Assume that $\vec{p}\in\PShearerSetInterior{G}$. Define the vector $\vec{x}$ by
\begin{subequations}
\begin{equation}\label{eq:intrinsicVector}
	\ForAll v\in V:\qquad
	x_v:=\inf\Set{\ShearerOVOEP{W}{v}(\vec{p}): (W,v)\text{ is escaping}}\,.
\end{equation}
Then
\begin{equation}\label{eq:intrinsicVectorMinoration}
	\ShearerMeasure{G}{\vec{p}}
	\Stochastically{\ge}
	\BernoulliProductField{\vec{x}}{V}\,,
	\qquad\text{ that is }
	\vec{x}\in\DominatedVectors{\ShearerMeasure{G}{\vec{p}}}\,,
\end{equation}
and
\begin{equation}
	\ForAll v\in V:\quad
	x_v\ge\min\Set{q_w:w\in\Neighbours{v}}>0\,.
\end{equation}

\end{subequations}
\end{Prop}

\begin{proof}
This is the same proof as the one for proposition \ref{prop:shearerImpliesUniformDominationInfinite}, except that instead of using the auxiliary BPF $X$ and \eqref{eq:dominationConditionalMinoration} from proposition \ref{prop:dominationConditionalMinoration} we use \eqref{eq:intrinsicConditionalMinoration} from proposition \ref{prop:intrinsicConditionalMinoration} directly.
\end{proof}

\begin{Prop}\label{prop:tightnessOnZ2}
Let $G:=\IntNum^2$. Define the sets
\begin{subequations}\label{eq:tightnessOnZ2}
\begin{equation}\label{eq:tightnessOnZ2Set}
	W_{(n,k,l)}:=\Set{(x,y): 0\le x< n, 0\le y<k+l}\uplus\Set{x=n,0\le y<k}
\end{equation}
and the value
\begin{equation}\label{eq:tightnessOnZ2Value}
	a(p):=\inf\Set{\ShearerOVOEP{W_{(n,k,l)}}{(n,k)}(p)}\ge q\,.
\end{equation}
Then, with $G_N$ being the subgraph induced by $V_N:=\Set{(x,y): 0\le x,y<N}$, we have
\begin{equation}\label{eq:tightnessOnZ2Limit}
	\lim_{N\to\infty} \frac{\log\ShearerCriticalFunction{G_N}(p)}{N^2} = \log a(p)\,.
\end{equation}
This implies that
\begin{equation}\label{eq:tightnessOnZ2DominatedParameter}
	\DominatedValue{\ShearerMeasure{\IntNum^2}{p}}
	=a(p)
	\ge\QShearer{\IntNum^2}
	>0\,.
\end{equation}
\end{subequations}
\end{Prop}

\begin{Rem}
The result of proposition \ref{prop:tightnessOnZ2} should be easily generalizable to $\IntNum^d$ and other $d$-dimensional transitive lattice like graphs. I even go so far as to conjecture that something similar should hold on all infinite, locally finite quasi-transitive graphs with quasi-transitive parameters. The obstacle seems mostly notational, especially in writing down a nice subset of escaping $(W,v)$s exhausting $V$.\\

Contrast the proof of proposition \ref{prop:tightnessOnZ2} with the subadditive approach in \cite[section $8.3$]{ScottSokal_repulsive}.
\end{Rem}

\begin{proof}
We see that for $(n,k,l)\le (\bar{n},\bar{k},\bar{l})$ we have
\begin{equation*}
	\ShearerOVOEP{W_{(n,k,l)}}{(n,k)}(p)
	\ge\ShearerOVOEP{W_{(\bar{n},\bar{k},\bar{l})}}{(\bar{n},\bar{k})}(p)
\end{equation*}
and hence
\begin{equation*}
	a(p)=\lim_{n,k.l\to\infty} \ShearerOVOEP{W_{(n,k,l)}}{(n,k)}(p)
\end{equation*}
is well-defined.\\

Enumerate $V_N=:\Set{v_1,\dotsc,v_{N^2}}$ starting from $(\Floor{\frac{N}{2}},\Floor{\frac{N}{2}})$ stepping to the left for the first step and spiraling outwards anti-clockwise around $(\Floor{\frac{N}{2}},\Floor{\frac{N}{2}})$. Setting $W_i:=\Set{v_1,\dotsc,v_{i-1}}$ we have
\begin{equation*}
	\ShearerCriticalFunction{G_N}(p)
	=\prod_{i=1}^{N^2} \ShearerOVOEP{W_i}{v_i}(p)
	=\prod_{i=1}^{N^2} \ShearerOVOEP{W_{(n_i,k_i,l_i)}}{(n_i,k_i)}(p)\,.
\end{equation*}

Thus on the one hand we have
\begin{equation*}
	\ShearerCriticalFunction{G_N}(p)\ge a(p)^{N^2}\,.
\end{equation*}

On the other hand choose $\varepsilon>0$. Then there exists a $(n,k.l)$ such that $\ShearerOVOEP{W_{(n,k,l)}}{(n,k)}(p)\le a(p)+\varepsilon$. We estimate roughly
\begin{multline*}
	\ShearerCriticalFunction{G_N}(p)
	= \ShearerCriticalFunction{\Graph{W_{(n,k,l)}}}(p)
		\prod_{i=(n\lor(k+l+1))^2}^{N^2}
			\ShearerOVOEP{W_{(n_i,k_i,l_i)}}{(n_i,k_i)}(p)\\
	\le \ShearerCriticalFunction{\Graph{W_{(n,k,l)}}}(p)
		\left(a(p)+\varepsilon\right)^{N^2-(n\lor(k+l+1))^2-4N(k+l)}\,.
\end{multline*}
Therefore
\begin{equation*}
	a(p)^{N^2}
	\le \ShearerCriticalFunction{G_N}(p)
	\le \ShearerCriticalFunction{\Graph{W_{(n,k,l)}}}(p)
		\left(a(p)+\varepsilon\right)^{N^2-(n\lor(k+l+1))^2-4N(k+l)}
\end{equation*}
and
\begin{multline*}
	\log a(p)
	\le \frac{\log \ShearerCriticalFunction{G_N}(p)}{N^2}\\
	\le \frac{\log \ShearerCriticalFunction{\Graph{W_{(n,k,l)}}}(p)}{N^2}
	+ \frac{N^2-(n\lor(k+l+1))^2-4N(k+l)}{N^2}\log(a(p)+\varepsilon)\,,
\end{multline*}
resulting in \eqref{eq:tightnessOnZ2Limit} by taking the limit.\\

For \eqref{eq:tightnessOnZ2DominatedParameter} adapt the reasoning from proposition \ref{prop:intrinsicDominationMinoration} to the escaping $\ShearerOVOEP{W_{(n,k,l)}}{(n,k)}(p)$, do the coupling from \eqref{eq:intrinsicCoupling} and the estimate from theorem \ref{thm:LLLClassical} to get a lower bound $\DominatedValue{\ShearerMeasure{\IntNum^2}{p}}\ge\DominatedValue{\ShearerMeasure{\IntNum^2}{\PShearer{\IntNum^2}}}\ge\QShearer{\IntNum^2}>0$. The upper bound follows straight from \eqref{eq:tightnessOnZ2Limit} (use the monotone functions $f_W:=\Indicator{\vec{1}_W}$).
\end{proof}
\subsection{Tools for stochastic domination II}
\label{sec:toolsForStochasticDominationTwo}
This section contains a number of proofs omitted in section \ref{sec:toolsForStochasticDomination} as well as some additional comments regarding stochastic domination of BRFs.\\

For $W\subset V$ and $\vec{s}_W\in\Configurations{W}$ we define the \emph{cylinder set} $\Cylinder{W}{\vec{s}_W}$ by
\begin{equation}\label{eq:SOCcylinder}
	\Cylinder{W}{\vec{s}_W}
	:= \Set{\vec{t}\in\Configurations{V}: \vec{t}_W=\vec{s}_W}\,.
\end{equation}

\begin{Lem}[{\cite[chapter II, theorem 2.4]{Liggett_ips}}]
\label{lem:stochasticDominationCouplingEquivalence}
Let $Y,Z$ be two BRFs indexed by $V$, then $Y\Stochastically{\ge}Z$ iff there exists a $\nu\in\ProbabilityMeasureSpace{\Configurations{V}^2}$ such that
\begin{subequations}\label{eq:stochasticDominationCouplingEquivalence}
\begin{gather}
	\ForAll\text{ finite }W\subseteq V,\ForAll \vec{s}_W\in\Configurations{W}:\quad
	\nu(\Cylinder{W}{\vec{s}_W}\times\Configurations{V})=\Proba(Y_W=\vec{s}_W)\\
	\ForAll\text{ finite }W\subseteq V,\ForAll \vec{t}_W\in\Configurations{W}:\quad
	\nu(\Configurations{V}\times\Cylinder{W}{\vec{t}_W})=\Proba(Z_W=\vec{t}_W)\\
	\label{eq:dominationAboveDiagonal}
	\nu(\Set{(\vec{s},\vec{t})\in\Configurations{V}^2: \vec{s}\ge\vec{t}})=1\,.
\end{gather}
\end{subequations}
\end{Lem}

\begin{Rem}
The coupling probability measure $\nu$ in lemma \ref{lem:stochasticDominationCouplingEquivalence} is in general not unique.
\end{Rem}

\begin{Prop}\label{prop:dominationBetweenBRFsImplications}
Let $Y$ and $Z$ be two BRFs indexed by the same set $V$. Then we have:
\begin{equation}\label{eq:dominationBetweenBRFsImplications}
	Y\Stochastically{\ge}Z
	\Then\ForAll\text{ finite }W\subseteq V:\quad
	\left(\begin{gathered}
		\Proba(Y_W=\vec{1})\ge\Proba(Z_W=\vec{1})\\
		\text{ and }\\
		\Proba(Y_W=\vec{0})\le\Proba(Z_W=\vec{0})
	\end{gathered}\right)\,.
\end{equation}
\end{Prop}

\begin{proof}
Assume that $Y\Stochastically{\ge}Z$ and let $W\subseteq V$ be finite. Lemma \ref{lem:dominationAndSubfields} asserts that $Y_W\Stochastically{\ge}Z_W$. Regard the monotone functions $f=\Indicator{\Cylinder{W}{\vec{1}}}$ and $g=1-\Indicator{\Cylinder{W}{\vec{0}}}$. Stochastic domination implies that
\begin{equation*}
	\Proba(Y_W=\vec{1})=\Expect[f(Y)]\ge\Expect[f(Z)]=\Proba(Z_W=\vec{1})
\end{equation*}
and
\begin{equation*}
	\Proba(Y_W=\vec{0})=1-\Expect[g(Y)]\le 1-\Expect[g(Z)]=\Proba(Z_W=\vec{0})\,.
\end{equation*}
\end{proof}

\begin{Prop}\label{prop:propertiesOfDominatedVectors}
Let $Y$ be a BRF taking values in $\Configurations{V}$. Then $\DominatedVectors{Y}$ is closed and a down-set.
\end{Prop}

\begin{proof}
Take a finite $W\subseteq V$. Then $\DominatedVectors{Y_W}$ is closed because we have a finite number of inequalities over the space of probability measures on $\Configurations{W}$, which is at most $2^{\Cardinality{W}}$-dimensional. If $\vec{c}\in\DominatedVectors{Y_W}$ and $\vec{d}\le\vec{c}$, then $\BernoulliProductField{\vec{d}}{W}\Stochastically{\le}\BernoulliProductField{\vec{c}}{W}\Stochastically{\le} Y$. Therefore $\DominatedVectors{Y_W}$ is a down-set. Those properties then carry over to $\DominatedVectors{Y}$ by taking the limit in the net of finite subsets of $V$.
\end{proof}

\begin{proof}(of \eqref{eq:dominationVWMM})
Take a finite $W\subseteq V$ and $f\in\MonotoneFunctionsOn{W}$. Then
\begin{align*}
	&\FirstAlign\Expect[f(Y_W\land Z_W)]\\
	&=\sum_{\vec{z}\in\Support Z_W}
		\Expect[f(Y_W\land\vec{z})|Z_W=\vec{z}]\Proba(Z_W=\vec{z})\\
	&\le\sum_{\vec{z}\in\Support Z_W}
		\Expect[f(Y_W)|Z_W=\vec{z}]\Proba(Z_W=\vec{z})\\
	&=\Expect[f(Y_W)]\\
	&=\sum_{\vec{z}\in\Support Z_W}
		\Expect[f(Y_W)|Z_W=\vec{z}]\Proba(Z_W=\vec{z})\\
	&\le\sum_{\vec{z}\in\Support Z_W}
		\Expect[f(Y_W\lor\vec{z})|Z_W=\vec{z}]\Proba(Z_W=\vec{z})\\
	&=\Expect[f(Y_W\lor Z_W)]\,.
\end{align*}
Hence $Y_W\land Z_W\Stochastically{\le}Y_W\Stochastically{\le}Y_W\lor Z_W$. For $\vec{x}\in\Configurations{W}$ and $f\in\MonotoneFunctionsOn{W}$ define
\begin{equation*}
	f_{\vec{x}}:\quad
	\Configurations{W}\to\RealNum\qquad
	\vec{y}\mapsto f(\vec{y}\land\vec{x})\,.
\end{equation*}
Then $f_{\vec{x}}\in\MonotoneFunctionsOn{W}$, as
\begin{equation*}
	\vec{y}\le\vec{z}
	\Then \vec{y}\lor\vec{x}\le\vec{z}\lor\vec{x}
	\Then f_{\vec{x}}(\vec{y})
		= f(\vec{y}\lor\vec{x})
	  \le f(\vec{z}\lor\vec{x})
	    = f_{\vec{z}}(\vec{y})\,.
\end{equation*}
We get
\begin{align*}
	&\FirstAlign\Expect[f(Y_W\lor X_W)]\\
	&=\sum_{\vec{x}\in\Configurations{W}}
		\Expect[f(Y_W\land\vec{x})]\Proba(X_W=\vec{x})\\
	&=\sum_{\vec{x}\in\Configurations{W}}
		\Expect[f_{\vec{x}}(Y_W)]\Proba(X_W=\vec{x})\\
	&\ge\sum_{\vec{x}\in\Configurations{W}}
		\Expect[f_{\vec{x}}(Z_W)]\Proba(X_W=\vec{x})
		&\text{as }Y_W\Stochastically{\ge}Z_W\text{ and }f\in\MonotoneFunctionsOn{W}\\
	&=\Expect[f(Z_W\lor X_W)]\,.
\end{align*}
The same derivation holds for $\land$ instead of $\lor$. Note that the fact that $X$ is independent of $(Y,Z)$ is crucial, as we do not know if $Y_W|X=\vec{x}\Stochastically{\ge}Z_W|X=\vec{x}$. Finally \eqref{eq:dominationVWMM} results from applying lemma \ref{lem:dominationAndSubfields}.
\end{proof}

\begin{proof}(of proposition \ref{prop:russoInhomogeneousExtension})
We show that $\nu$ fulfills the conditions of \eqref{eq:dominationAndSubfields}. During this proof we interpret $\IntegerSet{0}$ as $\emptyset$. We define a probability measure $\nu$ on $\Configurations{\NatNum^2}$ inductively by:
\begin{multline*}
  \ForAll n\ge 1,
  \ForAll \vec{s}_{\IntegerSet{n-1}},\vec{t}_{\IntegerSet{n-1}}
          \in\Configurations{\IntegerSet{n-1}},
  \ForAll a,b\in\Set{0,1}:\\
  \nu(\Cylinder{\Set{n}}{a}\times\Cylinder{\Set{n}}{b}\,|\,
      \Cylinder{\IntegerSet{n-1}}{\vec{s}_{\IntegerSet{n-1}}}\times
      \Cylinder{\IntegerSet{n-1}}{\vec{t}_{\IntegerSet{n-1}}})\\
  :=\begin{cases}
    =\Proba(Z_{n}=1|Z_{\IntegerSet{n-1}}=\vec{s}_{\IntegerSet{n-1}})
      &\text{if }(a,b)=(1,1)\\
    = 0
      &\text{if }(a,b)=(1,0)\\
    = p_n - \Proba(Z_{n}=1|Z_{\IntegerSet{n-1}}=\vec{s}_{\IntegerSet{n-1}})
      &\text{if }(a,b)=(0,1)\\
    = 1 - p_n
      &\text{if }(a,b)=(0,0)\,.
\end{cases}
\end{multline*}
A straightforward induction over $n$ shows that $\nu$ is a probability measure. The induction base is
\begin{equation*}
	\sum_{s_1,t_1}\nu(\Cylinder{\Set{1}}{s_1}\times\Cylinder{\Set{1}}{t_1})
	=(1-p_1)+(p_1-\Proba(Z_1=1))+0+\Proba(Z_1=1)=1\,.
\end{equation*}
The induction step is
\begin{align*}
	&\FirstAlign\sum_{\vec{s}_{\IntegerSet{n}},\vec{t}_{\IntegerSet{n}}}
		\nu(\Cylinder{\IntegerSet{n}}{\vec{s}_{\IntegerSet{n}}}
			\times
			\Cylinder{\IntegerSet{n}}{\vec{t}_{\IntegerSet{n}}})\\
	&=\sum_{\vec{s}_{\IntegerSet{n-1}},\vec{t}_{\IntegerSet{n-1}}}
		\nu(\Cylinder{\IntegerSet{n}}{\vec{s}_{\IntegerSet{n}}}
			\times
			\Cylinder{\IntegerSet{n}}{\vec{t}_{\IntegerSet{n}}})\\
	&\times\underbrace{\left(\sum_{s_n,t_n}
		\nu(\Cylinder{\Set{n}}{s_n}\times\Cylinder{\Set{n}}{t_n}\,|\,
		\Cylinder{\IntegerSet{n-1}}{\vec{s}_{\IntegerSet{n-1}}}\times
		\Cylinder{\IntegerSet{n-1}}{\vec{t}_{\IntegerSet{n-1}}})
		\right)}_{=1\text{ by definition of }\nu}\\
	&=\underbrace{\sum_{\vec{s}_{\IntegerSet{n-1}},\vec{t}_{\IntegerSet{n-1}}}
		\nu(\Cylinder{\IntegerSet{n}}{\vec{s}_{\IntegerSet{n}}}
			\times
			\Cylinder{\IntegerSet{n}}{\vec{t}_{\IntegerSet{n}}})}_{=1\text{ by induction}}\,.
\end{align*}
Next we calculate its marginals. Let $n\ge 1$ and $\vec{s}_{\IntegerSet{n}}\in\Configurations{\IntegerSet{n}}$. Then we have
\begin{align*}
  &\FirstAlign \nu(\Cylinder{\IntegerSet{n}}{\vec{s}_{\IntegerSet{n}}}
             \times\Configurations{\NatNum})\\
  &= \prod_{i=1}^n
       \nu(\Cylinder{\Set{i}}{\vec{s}_i}\times\Configurations{\NatNum}\,|\,
           \Cylinder{\IntegerSet{i-1}}{\vec{s}_{\IntegerSet{i-1}}}
           \times\Configurations{\NatNum})\\
  &= \prod_{i=1}^n
       \Proba(Z_{i}=s_i|Z_{\IntegerSet{i-1}}=\vec{s}_{\IntegerSet{i-1}})\\
  &= \Proba(Z_{\IntegerSet{n}}=\vec{s}_{\IntegerSet{n}})
\end{align*}
and
\begin{align*}
  &\FirstAlign
     \nu(\Configurations{\NatNum}
         \times\Cylinder{\IntegerSet{n}}{\vec{s}_{\IntegerSet{n}}})\\
  &= \prod_{i=1}^n
     \nu(\Configurations{\NatNum}\times\Cylinder{\Set{i}}{s_i}\,|\,
		 \Configurations{\NatNum}\times
         \Cylinder{\IntegerSet{i-1}}{\vec{s}_{\IntegerSet{i-1}}})\\
	&= \prod_{i=1}^n\left[
		(1-p_i)\Indicator{\Set{0}}(s_i)+ p_i\Indicator{\Set{1}}(s_i)
		\right]\\
	&= \Proba(X_{\IntegerSet{n}}=\vec{s}_{\IntegerSet{n}})\,.
\end{align*}
Hence the marginal of the first coordinate has the same law as $Z$ and the marginal of the second coordinate has the law $\BernoulliProductField{\vec{p}}{\NatNum}$.\\

Finally we calculate \eqref{eq:dominationAboveDiagonal} for $\nu$. We proceed by induction over $n$. The induction base is
\begin{equation*}
	\nu(\Set{(\vec{s},\vec{t})\in\Configurations{V}^2: \vec{s}_1\ge\vec{t}_1})
	=\nu(\Set{(\vec{s},\vec{t})\in\Configurations{V}^2: \vec{s}_1=0<\vec{t}_1=1})
	=0\,.
\end{equation*}
The induction step is
\begin{align*}
	&\FirstAlign \nu(\Set{(\vec{s},\vec{t})\in\Configurations{V}^2:
		\vec{s}_{\IntegerSet{n}}\ge\vec{t}_{\IntegerSet{n}}})\\
	&= \underbrace{\nu(\Set{(\vec{s},\vec{t})\in\Configurations{V}^2:
		\vec{s}_{\IntegerSet{n-1}}\ge\vec{t}_{\IntegerSet{n-1}}})
		}_{=1\text{ by induction}}\\
	&\times\left(
		1-\underbrace{
			\nu(\Set{(\vec{s},\vec{t})\in\Configurations{V}^2:
				\vec{s}_n=0<\vec{t}_n=1\,|\,
				\vec{s}_{\IntegerSet{n-1}}\ge\vec{t}_{\IntegerSet{n-1}}}
			)
		}_{=0\text{ by definition of }\nu}
		\right)\\
	&= 1\,.
\end{align*}
Hence
\begin{equation*}
	\ForAll n\in\NatNum:\quad
	\nu(\Set{(\vec{s},\vec{t})\in\Configurations{V}^2:
		\vec{s}_{\IntegerSet{n}}\not\ge\vec{t}_{\IntegerSet{n}}})=0\,.
\end{equation*}
This implies that
\begin{equation*}
	\nu(\Set{(\vec{s},\vec{t})\in\Configurations{V}^2:
		\vec{s}\not\ge\vec{t}})=0\,.
\end{equation*}
\end{proof}
\subsection{A summary of the homogeneous case}
\label{sec:homogeneousCase}
In the homogeneous case each of the sets defined in \eqref{eq:baseDefinitionsInhomogeneous}, after being identified with the respective cross-sections, reduces to a one-dimensional interval described by its non-trivial endpoint. The \emph{dominated Bernoulli parameter value} (short: dominated value) of a BPF $Y$ is
\begin{subequations}\label{eq:baseDefinitionsHomogeneous}
\begin{equation}\label{eq:dominatedBernoulliValue}
	\DominatedValue{Y}
	:=\max\Set{c: Y\Stochastically{\ge}\BernoulliProductField{c}{V}}\,.
\end{equation}
\begin{NoteToSelf}
For finite $W$ we have $\DominatedValue{Y_W}=\max\Set{c: Y_W\Stochastically{\ge}\BernoulliProductField{c}{W}}$. As $[0,\DominatedValue{Y}]=\bigcap_{\text{finite }W\subseteq V} [0,\DominatedValue{Y_W}]$ the value $\DominatedValue{Y}$ is contained in all of them and a $\max$.
\end{NoteToSelf}

For a non-empty class $C$ of BRFs this extends to
\begin{equation}\label{eq:dominatedClassValue}
	\DominatedValue{C}
	:= \inf\Set{\DominatedValue{Y}: Y\in C}\,.
\end{equation}
\begin{NoteToSelf}
Whereas in \eqref{eq:dominatedClassValue} I have an $\inf$ instead of a $\min$, because I can not guarantee that this value is attained for some BRF $Y$.
\end{NoteToSelf}

The \emph{critical domination values} of a class $C$, assuming that $C(p)$ is non-empty for all $p$, are written as
\begin{equation}\label{eq:guaranteValueNonuniform}
	\PDomination{C}
	:= \inf\Set{p\in[0,1]: \ForAll Y\in C(p): \DominatedValue{Y}>0}
\end{equation}
and
\begin{equation}\label{eq:guaranteValueUniform}
	\PUniformDomination{C}
	:= \inf\Set{p\in[0,1]: \DominatedValue{C(p)}>0}\,.
\end{equation}
As the function $p\mapsto\DominatedValue{C(p)}$ is non-decreasing \eqref{eq:dominatedParameterMonotonicity} the sets $]\PDomination{C},1]$ and $]\PUniformDomination{C},1]$ are up-sets and we have the inequality
\begin{equation}\label{eq:guaranteValueInequality}
	\PDomination{C}\le\PUniformDomination{C}\,.
\end{equation}
\end{subequations}

The first known result is a bound on $\PUniformDomination{\DependencyClassWeak{G}}$ in the homogeneous case, only depending on the maximal degree of $G$:

\begin{Thm}[{\cite[theorem 1.3]{LiggettSchonmannStacey_domination}}]
\label{thm:LSSDominationUpperBound}
If $G$ has uniformly bounded degree by a constant $D$, then
\begin{subequations}
\begin{equation}\label{eq:pDominationUpperBound}
	\PUniformDomination{\DependencyClassWeak{G}}\le 1-\PowerFracDual{(D-1)}{D}
\end{equation}
and for $p\ge 1-\PowerFracDual{(D-1)}{D}$ the dominated parameter is uniformly minorated:
\begin{equation}\label{eq:dominatedValueLowerBound}
	\DominatedValue{\DependencyClassWeak{G}(p)}\ge
	\left(1-\left(\frac{q}{(D-1)^{(D-1)}}\right)^{1/D}\right)
	\left(1-\left(q(D-1)\right)^{1/D}\right)\,.
\end{equation}
Additionally
\begin{equation}\label{eq:dominatedValueConvergenceToOne}
\lim_{p\to 1} \DominatedValue{\DependencyClassWeak{G}(p)} = 1\,.
\end{equation}
\end{subequations}
\end{Thm}

Recall that for $k\in\NatNumZero$ the \emph{$k$-fuzz of $G=(V,E)$} is the graph with vertices $V$ and an edge for every pair of vertices at distance less than or equal to $k$ in $G$. Denote the $k$-fuzz of $\IntNum$ by $\FuzzKZ$. Note that $\FuzzKZ$ is $2k$-regular. As $\FuzzKZ$ has a natural order inherited from $\IntNum$ theorem \ref{thm:LSSDominationUpperBound} can be improved considerably:

\begin{Thm}[{\cite[theorems 0.0, 1.5 and corollary 2.2]{LiggettSchonmannStacey_domination}}]
\label{thm:LSSDominationKFuzzZ}
On $\FuzzKZ$ we have
\begin{subequations}
\begin{equation}\label{eq:pDominationKFuzzZ}
	\PDomination{\DependencyClassWeak{\FuzzKZ}}
	=\PUniformDomination{\DependencyClassWeak{\FuzzKZ}}
	=\PDomination{\DependencyClassStrong{\FuzzKZ}}
	=\PUniformDomination{\DependencyClassStrong{\FuzzKZ}}
	=1-\PowerFracDual{k}{(k+1)}\,.
\end{equation}
For $p\ge\PUniformDomination{\DependencyClassStrong{\FuzzKZ}}$ the dominated parameter is minorated by
\begin{equation}\label{eq:dominatedValueLowerBoundKFuzzZ}
	\DominatedValue{\DependencyClassWeak{\FuzzKZ}(p)}\ge
	\left(1-\left(\frac{q}{k^k}\right)^{\frac{1}{k+1}}\right)
	\left(1-\left(qk\right)^{\frac{1}{k+1}}\right)\,.
\end{equation}
This implies a jump of $\DominatedValue{\DependencyClassWeak{\FuzzKZ}(.)}$ at the critical value $\PUniformDomination{\DependencyClassWeak{\FuzzKZ}}$, namely
\begin{equation}\label{eq:dominatedValueJumpKFuzzZ}
	\ForAll k\in\NatNumZero:\quad\frac{k}{(k+1)^2}
	\le\DominatedValue{
		\DependencyClassWeak{\FuzzKZ}
		(\PUniformDomination{\DependencyClassWeak{\FuzzKZ}})
	}\,.
\end{equation}
\end{subequations}
\end{Thm}

To arrive at the equality in \eqref{eq:pDominationKFuzzZ} Liggett, Schonmann \& Stacey derived a lower bound from a particular probability measure, called \NameShearersMeasure{} (see section \ref{sec:shearerPrimer}). Furthermore it allowed them to show that
\begin{equation}
	\ForAll k\in\NatNumZero:\quad
	\DominatedValue{
		\DependencyClassStrong{\FuzzKZ}
		(\PUniformDomination{\DependencyClassStrong{\FuzzKZ}})
	}
	\le\frac{k}{k+1}\,.
\end{equation}

Thus our main result can be written as a corollary of theorems \ref{thm:shearerDominationEquivalenceInhomogeneous} and \ref{thm:uniformlyDominatedVector}:

\begin{Thm}\label{thm:shearerDominationEquivalenceHomogeneous}
Let $G$ be a locally finite and connected graph. Then
\begin{subequations}\label{eq:shearerDominationEquivalenceHomogeneous}
\begin{equation}\label{eq:shearerPEqualsDominationP}
	\PDomination{\DependencyClassWeak{G}}
	=\PUniformDomination{\DependencyClassWeak{G}}
	=\PDomination{\DependencyClassStrong{G}}
	=\PUniformDomination{\DependencyClassStrong{G}}
	=\PShearer{G}\,.
\end{equation}
If $G$ contains at least one infinite connected component and has uniformly bounded degree, then
\begin{equation}\label{eq:dominatedValueJumpInfinite}
	\DominatedValue{
		\DependencyClassWeak{G}
		(\PUniformDomination{\DependencyClassWeak{G}})
	}
	\ge (\QUniformDomination{\DependencyClassWeak{G}})^2
	> 0\,,
\end{equation}
whereas if $G$ is finite we have
\begin{equation}\label{eq:dominatedValueJumpFinite}
	\DominatedValue{
		\DependencyClassWeak{G}
		(\PUniformDomination{\DependencyClassWeak{G}})
	}
	=0\,.
\end{equation}
\end{subequations}
\end{Thm}

\begin{figure}[!htbp]
\begin{displaymath}
	\MxyDominationProofInequalities{0.5cm}{0.5cm}
\end{displaymath}

\caption[Inequalities in the proof of \eqref{eq:shearerDominationEquivalenceHomogeneous}]{Inequalities in the proof of \eqref{eq:shearerDominationEquivalenceHomogeneous}. The four center inequalities follow straight from \eqref{eq:guaranteValueInequality} and \eqref{eq:strongIsWeak}. The inequality \LabelNonShearerImpliesNonDomination{} is an adaption of the approach used for $\FuzzKZ$ in \cite{LiggettSchonmannStacey_domination}, while inequality \LabelShearerImpliesUniformDomination{} is the novel interpretation of the optimal bounds of \NameShearersMeasure{}.}

\label{fig:structureOfInequalitiesInShearerDominationEquivalenceHomogeneous}
\end{figure}

The discontinuity described in \eqref{eq:dominatedValueJumpInfinite} also holds for the more esoteric case of graphs having no uniform bound on their degree. In this case $\PUniformDomination{\DependencyClassStrong{G}}=1$ and $\DominatedValue{\DependencyClassWeak{G}(1)}=1>0$. An explanation for this discontinous transition might come from statistical mechanics, via the connection with hard-core lattice gases made by Scott \& Sokal \cite{ScottSokal_repulsive}. It should be equivalent to the existence of a non-physical singularity of the entropy for negative real fugacities for all infinite connected lattices.\\

The graph $\FuzzKZ$ turns out to be a rare example of an infinite graph where we can construct \NameShearersMeasure{} explicitely, in this case as a $(k+1)$-factor \cite[section 4.2]{MathieuTemmel_kindependent}. A second case immediately deducible from previous work would be the $D$-regular tree $\Tree_D$, where
\begin{equation*}
	1-\PowerFracDual{(D-1)}{D}
	= \PShearer{\Tree_D}
	\le\PDomination{\Tree_D}
	\le 1-\PowerFracDual{(D-1)}{D}
\end{equation*}
by \cite[theorem 2]{Shearer_problem} and theorem \ref{thm:LSSDominationUpperBound}.
\subsection{Proofs of classical results}
\label{sec:classicalResults}
The following proofs are given for completeness and to be able to underline the similarity with the stochastic domination proofs.\\

\begin{proof}(of lemma \ref{lem:shearerMinimality})
It is sufficient to prove \eqref{eq:shearerMinimalityOVOEP} inductively for one-vertex extensions. We prove \eqref{eq:shearerMinimality} jointly by induction over the cardinality of $W$. The induction base for $W:=\Set{w}$ is
\begin{equation*}
	\Proba(Z_w=1)
	= p_w
	= \ShearerMeasure{G}{\vec{p}}(Y_w=1)
	=\ShearerCriticalFunction{(\Set{w},\emptyset)}(\vec{p})\,.
\end{equation*}
In the induction step we extend $W$ to $\widetilde{W}:=W\uplus\Set{v}$. Suppose that $\ShearerMeasure{G}{\vec{p}}(Y_W=\vec{1})=0$. Hence $\ShearerMeasure{G}{\vec{p}}(Y_{\widetilde{W}}=\vec{1})=0$, too, and \eqref{eq:shearerMinimalityCriticalFunction} holds trivially. If $\ShearerMeasure{G}{\vec{p}}(Y_{W}=\vec{1})>0$, then $\Proba(Z_{W}=\vec{1})>0$ by the induction hypothesis. Let $W\cap\Neighbours{v}=:\Set{w_1,\dotsc,w_m}$ and $W_i:=W\setminus\Set{w_i,\dotsc,w_m}$. If $m=0$, then we revert to the equality in the induction base. If $m\ge 1$ then
\begin{align*}
	&\FirstAlign\Proba(Z_v=1|Z_W=\vec{1})\\
	&=\frac{\Proba(Z_v=1,Z_W=\vec{1})}{\Proba(Z_W=\vec{1})}\\
	&\ge \frac%
		{\Proba(Z_W=\vec{1})-q_v\,\Proba(Z_{W\setminus\Neighbours{v}}=\vec{1})}
		{\Proba(Z_W=\vec{1})}
		&\text{as }Z\in\DependencyClassWeak{G}(\vec{p})\\
	&= 1 - \frac{q_v}{\prod_{i=1}^m \Proba(Z_{w_i}=\vec{1}|Z_{W_i}=\vec{1})}\\
	&\ge 1 - \frac{q_v}{\prod_{i=1}^m\ShearerOVOEP{W_i}{w_i}(\vec{p})}
		&\text{induction hypothesis as $\Cardinality{W_i}<\Cardinality{W}$ }\\
	&=\ShearerOVOEP{W}{v}(\vec{p})
	  &\text{using the fundamental identity \eqref{eq:shearerOVOEPFundamentalIdentity}}\\
\end{align*}
This proves \eqref{eq:shearerMinimalityOVOEP}. For \eqref{eq:shearerMinimalityCriticalFunction} see that
\begin{multline*}
  \Proba(Z_{\widetilde{W}}=\vec{1})
  =\Proba(Z_v=1|Z_W=\vec{1})\Proba(Z_W=\vec{1})\\
  \ge\ShearerOVOEP{W}{v}(\vec{p})\ShearerMeasure{G}{\vec{p}}(Y_W=\vec{1})
  =\ShearerMeasure{G}{\vec{p}}(Y_{\widetilde{W}}=\vec{1})\,.
\end{multline*}
\end{proof}

\begin{proof}(of theorem \ref{thm:LLLHomogeneousEscaping})
Assume that $q\le\PowerFracDual{(D-1)}{D}$. We claim that for every escaping $(W,v)$ (see definition \ref{def:shearerOVOEPEscaping})
\begin{equation}\label{eq:LLLHomogeneousEscapingInduction}
	\ShearerOVOEP{W}{v}(p) \ge 1-\frac{1}{D}\,.
\end{equation}
This claim implies that $\ShearerCriticalFunction{\Graph{W}}(p)\ge\left(\frac{D-1}{D}\right)^{\Cardinality{W}}>0$ for every finite $W\subseteq V$. Hence $p\ge\PShearer{G}$. We prove the claim \eqref{eq:LLLHomogeneousEscapingInduction} by induction over the cardinality of $W$. The induction base is given by
\begin{equation*}
	\ShearerOVOEP{\emptyset}{v}(p)
	=p
	\ge 1-\PowerFracDual{(D-1)}{D}
	\ge 1-\frac{1}{D}\,.
\end{equation*}
As $(W,v)$ is escaping $v$ has at most $m\le D-1$ neighbours in $W$, which we denote by $\Set{w_1,\dotsc,w_m}:=W\cap\Neighbours{v}$. Using the fundamental identity \eqref{eq:shearerOVOEPFundamentalIdentity} and \eqref{eq:LLLHomogeneousEscapingInduction} the induction step is
\begin{multline*}
	\ShearerOVOEP{W}{v}(p)
	= 1- \frac{q}{\prod_{i=1}^m \ShearerOVOEP{W\setminus\Set{w_i,\dotsc,w_m}}{w_i}(p)}\\
	\ge 1 - \frac{q}{\prod_{i=1}^m (1-\frac{1}{D})}
	\ge 1 - \frac{q}{\left(\frac{D-1}{D}\right)^{D-1}}\
	\ge 1 - \frac{1}{D}\,.
\end{multline*}
\end{proof}

\end{document}